\documentclass[11pt,reqno]{amsart}%

\usepackage[utf8]{inputenc}
\usepackage[T1]{fontenc}
\DeclareSymbolFont{rsfscript}{OMS}{rsfs}{m}{b}
\DeclareSymbolFontAlphabet{\mathrsfs}{rsfscript}
\usepackage{amsfonts,enumerate,amsmath,amssymb,amsthm}
\usepackage{pgf,pstricks}
\usepackage{colortbl}
\usepackage[absolute]{textpos}
\usepackage{stmaryrd}
\usepackage{subcaption}
\usepackage{dsfont}
\usepackage{float}
\usepackage{pstricks,tikz}
\usepackage{graphicx}
\usepackage{pst-grad}
\usepackage{multido}
\usepackage{tikz,tikzrput}
\usetikzlibrary{shapes,snakes}
\usetikzlibrary{positioning}
\usepackage{float}
\usepackage{ytableau}

\usepackage{geometry}
\newtheorem{theorem}{Theorem}[section]
\newtheorem{corollary}[theorem]{Corollary}

\newtheorem{lemma}[theorem]{Lemma}
\newtheorem{proposition}[theorem]{Proposition}

\theoremstyle{definition}
\newtheorem{definition}[theorem]{Definition}
\newtheorem{example}[theorem]{Example}
\newtheorem{remark}[theorem]{Remark}
\newtheorem{notation}[theorem]{Notation}

\newcommand\blfootnote[1]{%
  \begingroup
  \renewcommand\thefootnote{}\footnote{#1}%
  \addtocounter{footnote}{-1}%
  \endgroup
}

\newcommand{\Z}{\mathbb{Z}}
\newcommand{\N}{\mathbb{N}}
\newcommand{\C}{\mathbb{C}}
\newcommand{\R}{\mathbb{R}}
\newcommand{\K}{\mathbb{K}}

\newcommand{\fB}{\mathfrak{B}}
\newcommand{\cA}{\mathcal{A}}

\newcommand{\cR}{\mathcal{R}}
\newcommand{\cC}{\mathcal{C}}

\newcommand{\sZ}{\mathsf{Z}}
\newcommand{\sz}{\mathsf{z}}
\newcommand{\mB}{\mathsf{B}}
\newcommand{\mb}{\mathsf{b}}

\newcommand{\bt}{\mathbf{t}}

\newcommand{\Ga}{\Gamma}
\newcommand{\ga}{\gamma}
\newcommand{\be}{\beta}
\newcommand{\de}{\delta}
\newcommand{\la}{\lambda}

\newcommand{\func}[1]{\operatorname{#1}}
\newcommand{\qu}[1]{\quad\text{#1}\quad}

\newcommand{\und}[1]{\underline{#1}}

\newcommand{\si}{\sigma}

\newcommand{\om}{\omega}

\newcommand{\G}[1]{\mathsf{Graph}_{#1}(\K[\sZ^{\pm1}])}
\newcommand{\sG}[2]{\mathsf{Graph}_{#1}(#2)}

\newcommand{\sv}{\mathsf{v}}

\newcommand{\croc}[1]{\raisebox{0.1ex}{\scalebox{.8}{$[#1]$}}\vspace{-.1mm}}

\definecolor{ForestGreen}{cmyk}{0.91,0,0.88,0.42}

\numberwithin{equation}{section} 

\newcommand{\col}[2]{\mathsf{Col}_{#1}(#2)}
\newcommand{\row}[2]{\mathsf{Row}_{#1}(#2)}

\newcommand{\se}{\mathsf{e}}

\newcommand{\wt}{\mathrm{wt}}

\begin{document}

\title[Multiplicative graphs and algebras]{Basics on positively multiplicative graphs and algebras}
\author{J\'{e}r\'{e}mie Guilhot, C\'{e}dric Lecouvey and Pierre Tarrago}

\begin{abstract}
An oriented graph is said positively multiplicative when its adjacency matrix $A$ embeds
in a matrix algebra admitting a basis $\mB$ with nonnegative
structure constants in which the matrix of the multiplication by $A$ coincides with
$A$. The goal of this paper is to present basic properties of this notion and explain,
through various simple examples, how it relates to highly non trivial problems
like the combinatorial description of fusion rules, the description of the minimal boundary of graded
graphs or the study of random walks on alcove tilings.

\end{abstract}

\blfootnote{{\it 2020 Mathematics Subject Classification. 05,15,16,20\smallskip}}
\keywords{associative algebras, graphs, harmonic functions, fusion rules. \smallskip} 

\maketitle

\blfootnote{The three authors are supported by the Agence Nationale de la Recherche funding ANR CORTIPOM 21-CE40-001. 
The first author is supported by the Agence Nationale de la Recherche funding ANR JCJC Project ANR-18-CE40-0001 and by the Australian Research Council discovery project DP200100712.\smallskip}

\section{Introduction}

The aim of this paper is to study positively multiplicative algebras and positively
multiplicative graphs, two notions which provide a unified background to many
problems at the interaction between algebra, combinatorics and probability.
Given a finite set $\sZ$ of indeterminates and a field $\K$ equal to 
$\mathbb{R}$ or $\mathbb{C}$, a positively multiplicative algebra is a unital $\K[\sZ^{\pm1}]$-algebra $\mathcal{A}$ admitting a basis $\mB$ such
 that~$1\in\mB$ and the structure constants of $\cA$ with respect to $\mB$ lie in $\mathbb{R}_{+}[\sZ^{\pm1}]$,
the subset of $\K[\sZ^{\pm1}]$ of Laurent polynomials with nonnegative
real coefficients. The set $\sZ$ can be empty and then it is just required that the structure constants are nonnegative. This class of algebras contains the fusion
algebras \cite{Fuc}, the group algebras or the character algebras associated to finite
groups, the character algebras of simple Lie algebras,  but also homology or
cohomology rings defined from algebraic varieties related in particular to
Schubert calculus or to the geometry of affine Grassmannians, see \cite{LM,LLMSSZ}.

\medskip

Due to the positivity of the structure constants of $\cA$ with respect to $\mB$, the matrix
of the multiplication by any positive linear combination $s$ of elements in
$\mB$ expressed in the basis $\mB$ can be considered as an
adjacency matrix of an oriented weighted directed graph with set of vertices in bijection with the set
$\mB$. Conversely, starting from a finite oriented graph $\Gamma$ with weights
 in $\mathbb{R}_{+}[\sZ^{\pm1}]$, it is a natural question to look for
an underlying positively multiplicative algebra structure.\ We will say that
$\Gamma$ is positively multiplicative at a vertex $v_{i_{0}}$ when its
adjacency matrix~$A_\Ga$ can be embedded in a matrix algebra which is positively
multiplicative with respect to a basis~$\mB$ such that (1) the $i_0$-th element of $\mB$ is equal to the identity and (2) the matrix of the multiplication by
$A_\Ga$ expressed in the basis $\mB$ is the matrix $A_\Ga$ itself. One can notice here that we have two types of constraints for a graph to be
positively multiplicative which are of very different nature:
\begin{enumerate}
\item The first one relates to linear algebra: does there exist a matrix
algebra $\mathcal{A}$ with a basis~$\mB$ containing~$1$
and the matrix of the multiplication by $A$ is $A$? If so, we will say that~$\Gamma$
 is multiplicative.

\item The second one is of geometric nature: does there exist an algebra as
in (1) whose cone $\mathcal{C}(\mB)$ of nonnegative linear combinations
of elements in $\mB$ is stable by multiplication?
\end{enumerate}

We shall see that the answer to Question (1) is most of the time positive when the graph $\Gamma$ is of maximal dimension, that is when the algebra $\K[\sZ^{\pm1}][A_\Ga]$ has dimension the number of vertices in $\Ga$. In this case, the algebra $\mathcal{A}$ and its basis
$\mB$ are unique when they exist and can be computed by
elementary linear algebra techniques. In contrast, it is highly non trivial to
find general conditions sufficient to guarantee a positive answer to question
(2). Moreover, even when a graph $\Gamma$ is positively multiplicative, it is
often very difficult to get a combinatorial description of the structure
constants of the basis $\mB$, for example by counting paths in
$\Gamma$. Depending on the situation, this problem may be equivalent to the description of tensor product
multiplicities in representation theory, the determination of combinatorial
fusion rules or that of computing the structure constants in the homology rings of affine Grassmannians. For tensor product multiplicities of Kac-Moody algebras, an
elegant description exists in terms of the Littelmann path model \cite{Lit} or
the combinatorics of crystal graphs \cite{Kash}.\ For the symmetric groups,
these multiplicities (Kronecker coefficients) are much less understood. The
determination of combinatorial descriptions for the fusions rule in conformal
field theory \cite{Beauv} or the structure constants in the homology rings of affine Grassmannians \cite{LLMSSZ} is still an unsolved problem.

\medskip

The goal of the paper is to propose a unified approach in the study of
positively multiplicative graphs independent of the algebraic or geometric
context where they naturally appear. The results, presented in an expository
style, gather basic facts on these notions that we will need in future works and
for which we did not find explicit references in the literature. Such results will be used in particular in \cite{GLT2} and \cite{GLT3} to study
respectively random walks on alcoves tilings and convergences of random
particle systems on a discrete circle. Although
they use quite elementary tools, we believe that they deserved to be written
down. We will explain in particular, when the graph is of maximal dimension, how to decide whether a graph is
multiplicative and then provide simple procedures to compute the associated
basis.\ When a graph~$\Gamma$ is positively multiplicative, we
will also give a general construction (called expansion) yielding an infinite
graded graph~$\Gamma_{\se}$ defined from~$\Gamma$ for which it is easy to
get a complete description of the extremal positive harmonic functions.\ These
functions are essential tools in the study of random walks on graphs or
alcoves tilings (see for example \cite{Ker} and \cite{LT2}).

\medskip

In each of the following sections, we have chosen to present numerous examples
of the different notions we introduce. We hope they will be sufficiently helpful for the
reader. The paper itself is organised as follows.\ Section 2 is devoted to
the notion of positively multiplicative algebras and its connection with 
fusion algebras. 
 In Section 3, we define and study multiplicative
graphs~$\Gamma$. When $\Gamma$ is of  maximal dimension, we give a
simple procedure to check whether or not it is multiplicative at a given vertex and
then to compute its associated basis $\mB$. The positively
multiplicative graphs are presented in Section 4. 
In Section 5, we introduce  the column and row Kirillov-Reshetikhin crystals of affine type A. These are two important examples
of positively multiplicative graphs whose definition is elementary but with
structure constants of high combinatorial complexity. In Section 6, we detail the
expansion procedure and explain how to get a
description of the minimal boundary (i.e. of the extremal positive
harmonic functions) on each extended graph coming from a positively
multiplicative graph. Their generalisations
and their study will be our main objective in~\cite{GLT2}.
In Section 7, we study commutative positively multiplicative (possibly infinite dimensional) algebras over $\C$ and show that  
the subset of elements of the distinguished basis $\mB$ that are generalised permutations has the structure of an abelian group. In the case of a positively multiplicative algebra over $\C$ coming from a positively multiplicative graph, this subset is precisely the set of vertex $v$ of $\Ga$ for which $\Ga$ is positively multiplicative at $v$.

\section{Positively multiplicative algebras}
 We start by introducing some notation that we will use all along the paper:
\begin{enumerate}
\item[{\tiny $\bullet$}] $\K$ is either $\mathbb{R}$ or $\mathbb{C}$, 
\item[{\tiny $\bullet$}] $\sZ=\{z_{1},\ldots,z_N\}$ with $N\in \N$ is a set (possibly empty) of formal indeterminates,
\item[{\tiny $\bullet$}] $\K[\sZ^{\pm1}]$ is the Laurent
polynomial ring in $\sZ$ over $\K$,
\item[{\tiny $\bullet$}]  $\mathbb{R}_{+}[\sZ^{\pm1}]$ is the set of 
Laurent polynomials with nonnegative real coefficients,
\item[{\tiny $\bullet$}] $\K(\sZ^{\pm1})$ is the ring of fractions of $\K[\sZ^{\pm1}]$.
\end{enumerate}
 In this section, $\cA$ denotes a unital associative algebra over $\K[\sZ^{\pm1}]$.
\begin{definition}
\label{DefPMBases}The algebra $\mathcal{A}$ is said to be \emph{positively
multiplicative} (PM for short) when it admits a $\K[\sZ^{\pm1}]$-basis
$\mB=\{\mb_{i},i\in I\}$ indexed by a countable set $I$ satisfying the following two conditions:
\begin{enumerate}
\item $1\in\mB$,
\item for any $i,j\in I$, we have 
$
\mb_{i}\mb_{j}=\sum_{k\in I}c_{i,j}^{k}\mb_{k}$ with $c_{i,j}^{k}\in
\mathbb{R}_{+}[\sZ^{\pm1}].$
\end{enumerate}
When such a basis $\mB$ exists, we say that it is a positively
multiplicative basis for $\cA$ (PM-basis for short) and that $
\cA$ is positively multiplicative with respect to $\mB$. The algebra $\mathcal{A}$ is said to be \emph{strongly positively multiplicative} when it
is positively multiplicative and for any $(j,k)\in I^{2}$ there exists at
least one index $i$ in $I$ such that $c_{i,j}^{k}\neq0$.
\end{definition}

\begin{example}
\ \label{Ex_PMalgebra}

\begin{enumerate}
\item The algebra $\K^{n}$ is PM with respect to the basis 
 $\mB=\{\mb_{i}\mid i=1,\ldots,n\}$
with $\mb_{1}=(1,\ldots,1)$ and  $\mb_{i}=e_{i-1}$ for $i=2,\ldots,n$ where $(e_1,\ldots,e_n)$ denotes the canonical basis of $\K^n$. It is easy to construct other PM-basis for $\K^{n}$, for instance one could change $\mb_2$ to $(0,1,\ldots,1)$ in the basis $\mB$. In general a
PM basis is not unique.
\item Any finite-dimensional $\mathbb{C}$-algebra $\mathcal{A}$ generated by
one element is PM. Indeed, if we set $\mathcal{A}%
=\mathbb{C}[a]$ and denote by $\mu_{a}$ the minimal polynomial of $a$, we can
factorise~$\mu_{a}$
\[
\mu_{a}(X)=\prod_{i=1}^{k}(X-\alpha_{i})^{m_{i}}%
\]
where $\alpha_{1},\ldots,\alpha_{k}$ are the $k$ distinct complex roots of
$\mu_{a}$. Then, using the Chinese remainder theorem, we get%
\[
\mathcal{A}\cong\prod_{i=1}^{k}\mathbb{C}[X]/(X-\alpha_{i})^{m_{i}}.
\]
The algebra on the right hand side is PM because each algebra
$\mathbb{C}[X]/(X-\alpha_{i})^{m_{i}}$ is with respect to the basis $\{1,X-\alpha
_{i},\ldots,(X-\alpha_{i})^{m_{i}-1}\}$.

\item Regarded as a two-dimensional $\mathbb{R}$-algebra, the field
$\mathbb{C}$ of complex numbers is not positively multiplicative.
Indeed if such a basis $(1,z)$ existed, we could assume that $z^{2}=a+bz$ with
$(a,b)\in\mathbb{R}_{\geq0}^{2}$ and get that $z$ is a non real root of the
polynomial $X^{2}-bX-a$. Since its discriminant is positive, this yields a contradiction.

\item Let $\mathcal{A}$ be a finite-dimensional commutative subalgebra of
$\func{Mat}_{n}(\mathbb{C})$ (the algebra of $n\times n$ complex matrices) stable by the
adjoint operation. Then, each matrix $A$ in $\mathcal{A}$ is diagonalisable
because $AA^{\ast}=A^{\ast}A$. Moreover, since $\mathcal{A}$ is commutative,
there is a common basis of diagonalisation. Therefore, the algebra
$\mathcal{A}$ is isomorphic to $\mathbb{C}^{n}$ which is  PM.

\item Consider any polynomial $P(X)=X^{n}-a_{n-1}X^{n-1}-\cdots-a_{0}$ with
$(a_{n-1},\ldots,a_{0})\in\mathbb{R}_{+}^{n}$ and $A$ its companion
matrix. Then $\K[A]$ is a $n$-dimensional PM 
subalgebra of $\func{Mat}_{n}(\K)$ with respect to the basis $\{A^{k}%
\mid k=0,\ldots,n-1\}$.

\item For any finite group $G$, the group algebra $\mathbb{C}[G]$ is a PM
algebra with respect to the basis $\{g\mid g\in G\}$.

\item For any finite group $G$, its complex character ring $\mathcal{R}[G]$ is
a commutative PM algebra with respect to the basis of irreducible characters. Observe
that $\mathcal{R}[G]$ is also isomorphic to $\mathbb{C}^{n}$ just by considering
the basis of indicator functions associated to the conjugacy classes of $G$.

\item Homology and cohomology rings associated to algebraic varieties are
other important examples of PM algebras.
\end{enumerate}
\end{example}

Examples (6) and (7) above are particular cases of fusion algebras (see \cite[Sec. 5.1]{Fuc} for a detailed introduction to fusion algebras). A fusion algebra is an algebra over $\C$ (here we take
$\sZ=\emptyset$) with a positively multiplicative basis $\mB%
=\{\mb_{i},i\in I\}$ and an involutive anti-automorphism $\ast$ (that induces an involution $i\mapsto i^\ast$ on $I$ by setting $\mb_{i}^{\ast}=\mb_{i^{\ast}}$) that satisfy $c_{i,j}^{k}=c_{i^{\ast},k}^{j}$ for all $i,j,k\in
I$.
\begin{proposition}
Every fusion algebra is strongly positively multiplicative.
\end{proposition}

\begin{proof}
Let $\mathcal{A}$ be a fusion algebra with basis $\mB=\{\mb_{i},i\in I\}$. Since $\ast$
is an anti-automorphism, we have~$c_{i,j}^{k}=c_{j^{\ast},i^{\ast}}^{k^{\ast}}=c_{j,k^{\ast}}^{i^{\ast}}$ using the property of fusion algebras. Assume there exists $(j,k)\in
I^{2}$ such that $c_{i,j}^{k}=0$ for all $i\in I$. By the previous argument,
we have $c_{j,k^{\ast}}^{i^{\ast}}=0$ for all $i\in I$ and
thus~$c_{j,k^{\ast}}^{i}=0$ for all $i\in I$ since the map $i\mapsto i^{\ast}$
is a bijection on $I$. We get the equality $\mb_{j}\mb_{k^{\ast}}=0$. This
implies that~$\mb_{j^{\ast}}\mb_{j}\mb_{k^{\ast}}\mb_{k}=0$. Now observe that
$c_{j^{\ast},j}^{i_0}=c_{j,i_0}^{j}=1$ and similarly $c_{k^{\ast},k}^{i_0}=1$, where $\mb_{i_0}=1\in\mB$. Since
the structure constants of the fusion algebra
$\mathcal{A}$ are nonnegative, the coefficient of $1$ in the
product $\mb_{j^{\ast}}\mb_{j}\mb_{k^{\ast}}\mb_{k}$ is positive. Hence $\mb_{j^{\ast}}\mb_{j}\mb_{k^{\ast}}\mb_{k}\neq0$ and we get the desired contradiction.
\end{proof}

\begin{proposition}
\label{Lem_IntegralDomain}Let $\mathcal{A}$ be a finite-dimensional, commutative and positively multiplicative algebra. Then $\cA$ is integral over $\K[\sZ^{\pm1}]$.
\end{proposition}

\begin{proof}
Let $\mB=\{\mb_i,i\in I\}$ be a PM-basis for $\cA$. 
Since $\mathcal{A}$ is generated over $\K[\sZ^{\pm1}]$ by the elements of~$\mB$, it suffices to show that the
elements $\mb_{i}$ are integral over $\K[\sZ^{\pm1}]$.\ For any such 
$\mb_{i}$, the matrix $M$ of the multiplication by $\mb_{i}$ in
$\mathcal{A}$ in the basis $\mB$ has coefficients in $\K[\sZ^{\pm1}]$ (in fact in $\mathbb{R}_{+}[\sZ^{\pm1}]$)$.$ Thus the characteristic
polynomial of $M$ has coefficients in $\K[\sZ^{\pm1}]$. This implies
that its minimal polynomial, which is also the minimal polynomial of $\mb_{i}$
also has coefficients in $\K[\sZ^{\pm1}]$.\ Hence $\mb_{i}$ is
integral over $\K[\sZ^{\pm1}]$.
\end{proof}

\begin{remark}
\label{Remak_BBK}
Given a PM-algebra $\mathcal{A}$ with $\sZ\neq\emptyset$, any morphism
$\theta:\K[\sZ^{\pm1}]\rightarrow \K$ defines a specialisation of
$\mathcal{A}$ that we shall denote $\mathcal{A}_{\theta}.$ When $\theta
(\sz)\in\mathbb{R}_{>0}$ for any $\sz\in \sZ$, the algebra $\mathcal{A}%
_{\theta}$ remains positively multiplicative. Further if $\mathcal{A}$ is
strongly positively multiplicative, so is $\mathcal{A}_{\theta}$.
\end{remark}

\section{Multiplicative graph}
Let $\Ga=(V,E,\om)$ be a finite oriented weighted graph where $V$ is the finite set of vertices, $E$ is the set of edges $E$ and $\om$ is the edge weight function taking values in some algebra. We will assume that $\Ga$ does not have multiple edges, so that the set of edges can be viewed as a subset of $V\times V$: a pair $(v,v')\in V\times V$ represents an edge starting at $v$ and ending at $v'$. Finally, it will be convenient to extend the edge weight function $\om$ to the full set $V\times V$ simply by setting $\om(v,v') = 0$ if there is no edge from $v$ to $v'$. Note that with this setting, the sum of the weights of all edges starting at a given vertex $v$ is equal to $\sum_{v'\in V} \om(v,v')$.

\smallskip

Unless explicitely specified otherwise, all the graphs considered in this paper are finite oriented weighted graph with edge weight function taking values in $\K[\sZ^{\pm1}]$.
We denote by $\G{n}$ the set of all such graphs that contains $n$ vertices. In most examples and applications the indeterminates in $\sZ^{\pm1}$ will be eventually specialised to positive reals.

\smallskip

In this section, $\Ga$ denotes a graph in $\G{n}$ with set of vertices $\{v_1,\ldots,v_n\}$ and edge weight function $\om$. We denote by $A_\Ga=(a_{i,j})_{1\leq i,j\leq n}$ the adjacency matrix of $\Gamma$ (which lies in $\func{Mat}_{n}(\K[\sZ^{\pm1}])$),  that is the coefficient $a_{i,j}$ is equal to $\om(v_j,v_i)$.  The adjacency algebra $\cA_\Ga$ of $\Ga$ is the algebra $\K(\sZ^{\pm1})[A_\Ga]$ of polynomials in $A_\Ga$ with coefficients in~$\K(\sZ^{\pm1})$. Note that the adjacency matrix and the adjacency algebra of $\Ga$ are defined up to an ordering of the vertices of $\Ga$ and thus up to a conjugation of $A_\Ga$ by a permutation matrix.

\smallskip

In this section, we introduce the notion of multiplicative graph which lies at the heart of the paper.
We then study this notion in the case where the graph is of maximal dimension, that is when the adjacency algebra is of dimension the number of vertices of $\Ga$, where we give an explicit criterion to decide whether or not the graph is multiplicative.

\subsection{Multiplicative graph, roots and matrix realisation }

Given any finite dimensional associative algebra $\cA$, a basis $\mB$ of $\cA$ and $x\in \cA$, we denote by $m_x$ the endomorphism of $\cA$ defined by $a\mapsto xa$ and by $\func{Mat}_{\mB}(m_x)$ the matrix of $m_x$ in the basis $\mB$. We write $\mB\croc{i}$ for the $i$-th element of the basis~$\mB$.

\begin{definition}
\label{Def_MGraphs}
\begin{enumerate}
\item We say that $(\cA,\mB)$  is a matrix realisation of $\Ga$ if $\cA$ is a subalgebra of $\func{Mat}_n(\K(\sZ^{\pm1}))$ of dimension~$n$ containing $\cA_\Ga$ and $\mB$ is a basis of $\cA$ such that  $A_\Ga = \func{Mat}_{\mB}(m_{A_\Ga})$.
\item We say that $\Ga$ is \emph{multiplicative} at $v_{i_0}$ if there exists a matrix realisation $(\cA,\mB)$ of $\Ga$ with $\mB\croc{i_0} = I_n$. 
\item We say that $\Ga$ is \emph{multiplicative} if there exists a vertex $v_{i_0}$ such that $\Ga$ is \emph{multiplicative} at $v_{i_0}$.  
\item The set $\cR_\Ga$ of \emph{roots} of $\Ga$ is the set of vertices $v_{i_0}$ such that $\Ga$ is multiplicative at $v_{i_0}$.
\end{enumerate}
\end{definition}
We will see in Example \ref{Counterexample} that it is possible to have a graph $\Ga$ that admits a matrix realisation but which is not multiplicative. Note that if $(\cA,\mB)$  is a matrix realisation of $\Ga$ then $(\cA,\mB x)$ is also a matrix realisation of $\Ga$ for all invertible elements $x\in \cA^\times$. In particular, a graph that admits a matrix realisation $(\cA,\mB)$ such that there exists $\mb\in \mB\cap \cA^\times$ is multiplicative. Indeed the pair $(\cA,\mB\mb^{-1})$ is a matrix realisation of $\Ga$ and $\mB\mb^{-1}$ contains the identity.

\begin{example}
\label{mutl_graph_2}
Let $P=X^{n}-a_{n-1}X^{n-1}-\cdots-a_{0}\in\K[X]$ and let $A$ be the companion matrix of $P$. Let~$\Ga$ be the weighted graph with $n$ vertices $\{v_1,\ldots,v_{n}\}$ and adjacency matrix $A_\Ga$ equal to $A$ (here $\sZ=\emptyset$). For example, if $P= X^4-2X^3-X^2-3X-4$, the graph $\Ga$ and the matrix $A_\Ga$ are given by

$$
\begin{minipage}{8cm}
$$
\begin{tikzpicture}[scale =1.2]
\def\rx{1.2}
\def\ry{0.866}
\tikzstyle{vertex}=[inner sep=2pt,minimum size=10pt,ellipse,draw]

\node[vertex] (a0) at (0,0) {$v_1$};
\node[vertex] (a1) at (\rx,0) {$v_2$};
\node[vertex] (a2) at (2*\rx,0) {$v_3$};
\node[vertex] (a3) at (3*\rx,0) {$v_4$};

\draw[line width=.5pt,->] (a0) to   (a1);
\draw[line width=.5pt,->] (a1) to   (a2);
\draw[line width=.5pt,->] (a2) to   (a3);
\draw[line width=.5pt,->,bend left = 30] (a3) to  node[below]{\footnotesize{1}} (a2);
\draw[line width=.5pt,->,bend left = 60] (a3) to  node[below]{\footnotesize{3}}  (a1);
\draw[line width=.5pt,->,bend right = 30] (a3) to  node[above]{\footnotesize{4}}  (a0);

\draw[line width=.5pt,->] (a3) to[out=-35,in=45,looseness=10]  node[right]{\footnotesize{2}}  (a3);

\end{tikzpicture}
$$
\end{minipage}
\begin{minipage}{4cm}
$$
A_\Ga=\left(
\begin{array}{cccc}
0 & 0 & 0&4\\
1 & 0 & 0&3\\
0 & 1 & 0&1\\
0 & 0 & 1&2
\end{array}
\right)
.$$
\end{minipage}
$$
\begin{center}
\end{center}
The minimal polynomial of $A_\Ga$ is $P$ so that $\K[A]$ is of dimension $n$ and the pair $(\K[A],\mB)$ where $\mB=\{I_n,A_\Ga,\ldots,A_\Ga^{n-1}\}$ is a matrix realisation of $\Ga$. The graph $\Ga$ is multiplicative at $v_1$.
\end{example}

\begin{proposition}
\label{m_x}
The graph $\Ga$ is multiplicative at $v_{i_0}$ if and only if there exists a $n$-dimensional associative algebra $\cA$ over $\K(\sZ^{\pm 1})$, an element $x\in \cA$ and a basis $\mB=(\mb_1,\ldots,\mb_{n})$ of $\cA$ such that $\mb_{i_0} =1$ and $\func{Mat}_{\mB}(m_x) = A_\Ga$.
\end{proposition}
\begin{proof}
Assume that there exists such an algebra $\cA$. The map
$$\begin{array}{ccccccc}
\func{M}_{\mB}&:& \cA&\to&\func{Mat}_n(\K(\sZ^{\pm 1}))\\
&& y &\mapsto & \func{Mat}_{\mB}(m_y)
\end{array}$$
is an injective morphism of algebras. Then $\func{M}_{\mB}(\cA)$ is a subalgebra of $\func{Mat}_n(\K(\sZ^{\pm 1}))$. We claim that $(\func{M}_{\mB}(\cA),\mB)$ where $\mB = \{\func{M}_{\mB}(\mb_i)\mid i=1,\ldots,n\}$ is a matrix realisation of $\Ga$. First, $\func{M}_{\mB}(x)=A_{\Gamma}$ by construction, so that $\mathcal{A}_{\Gamma}\subset \func{M}_{\mB}(\cA)$. Then, we have $x\mb_j=m_x(\mb_j) = \sum_{i=1}^{n} a_{i,j} \mb_i$ so that $m_{x\mb_j} = \sum a_{i,j} m_{\mb_i}$. Hence, 
\begin{align*}
\func{M}_{\mB}(x) \func{M}_{\mB}(\mb_j)=A_\Ga \func{M}_{\mB}(\mb_j) = \func{Mat}_{\mB}(m_x) \func{Mat}_{\mB}(m_{\mb_j}) =\func{Mat}_{\mB}(m_xm_{\mb_j}) =\func{Mat}_{\mB}(m_{x\mb_j}) &= \sum_{i=1}^{n} a_{i,j} \func{Mat}_{\mB}(m_{\mb_i})\\
&=\sum_{i=1}^{n} a_{i,j} \func{M}_{\mB}(\mb_i).
\end{align*}
Since $\mb_{i_0}=1$, we have $\func{Mat}_{\mB}(m_{\mb_{i_0}}) =I_n$ and we get that $\Ga$ is multiplicative at $v_{i_0}$. The converse is obvious by definition of a multiplicative graph.
\end{proof}
\begin{example}
\label{mutl_graph_1}
Let $\Ga$ be the Cayley graph of the symmetric group $\mathfrak{S}_{n}$ associated to the transpositions $\tau_i = (1,i)$ with $i=2,\ldots,n$. That is, the set of vertices of $\Ga$ is $\{\sigma\mid \sigma\in \mathfrak{S}_n\}$ and there is an edge (with weight 1) between $\sigma$ and $\sigma'$ if and only if $\sigma' = \tau_i\sigma$ for some $i$. We show that $\Ga$ is multiplicative at $\se$ where~$\se$ denotes the identity of $\mathfrak{S}_{n}$. Let $\mathcal{A} = \K[\mathfrak{S}_{n}]$ be the group algebra of
the symmetric group $\mathfrak{S}_{n}$ with basis $\mB=\{e_{\sigma}\mid\sigma\in\mathfrak{S}_{n}\}$ and let $x=\sum_{2\leq i\leq n}e_{(1,i)}$. Then we have $A_\Ga = \func{Mat}_{\mB}(m_x)$ hence showing that $\Ga$ is multiplicative at $\se$. In fact, $\Ga$ is multiplicative at any vertex
$\rho$ by considering the basis $\mB e_{{\rho}^{-1}}$.
\end{example}

Given $i,j,k\in \{1,\ldots,n\}$, let  $m_{k\to i}^{j}\in \K[\sZ^{\pm1}]$ be the sum of the weights of all paths of length~$j$ that start at $v_k$ and finish at $v_i$. For $i_0\in \{1,\ldots,n\}$, we denote by $M_{i_0}$ the matrix with coefficients $(m_{i_0\to i}^{j})_{1\leq i,j\leq n}$.

\smallskip
Assume that $\Ga$ is multiplicative graph at $v_{i_0}$ and let $(\cA,\mB)$ be a matrix realisation of $\Ga$ where $\mB=\{\mb_1,\ldots,\mb_{n}\}$ and $\mb_{i_0}=I_n$.  By induction, for any integer $j\geq0$ and $k\in\{1,\ldots,n\}$ we have
\begin{equation}
A_\Ga^{j}\mb_{k}=\sum_{i=1}^{n}m_{k\to i}^{j}\mb_{i}. \label{Transfert}%
\end{equation}
 
We denote by $c_{i,j}^k$ where $i,j,k\in \{1,\ldots,n\}$ the structure constants of $\cA$ with respect to $\mB$, that is 
\[
\mb_{i}\mb_{j}=\sum_{k=1}^{n}c_{i,j}^{k}\mb_{k}.
\]

\begin{lemma}
\label{Lem_SPM_SC}
Assume that $\Gamma$ is strongly connected. Then, for all $1\leq j,k\leq n$, there exists at least one
$i\in\{1,\ldots,n\}$ such that $c_{i,j}^{k}\neq0$.
\end{lemma}

\begin{proof}
Assume that we have $1\leq j_{0},k_{0}\leq n$ such that $c_{i,j_{0}}^{k_{0}%
}=0$ for any $i\in\{1,\ldots,n\}$.\ This means that the right ideal $\mathcal{A}\mb_{j_{0}}$ generated by $\mb_{j_{0}}$ in $\mathcal{A}$ is contained in $\oplus_{k\neq
k_{0}}\K(\sZ^{\pm1})\mb_{k}$.\ Thus for any nonnegative integer~$\ell$, we have $A^{\ell}\mb_{j_{0}}\in\oplus_{k\neq k_{0}}\K(\sZ^{\pm1})\mb_{k}$. By
(\ref{Transfert}), this implies that there cannot exist a path in
$\Gamma$ from $v_{j_{0}}$ to~$v_{k_{0}}$ which contradicts our assumption that
$\Gamma$ is strongly connected.
\end{proof}

\subsection{Multiplicative graphs of maximal dimension}
Recall that the adjacency algebra $\cA_\Ga$ of $\Ga$ is defined to be $\K(\sZ^{\pm1})[A_\Ga]$ where $A_\Ga\in \func{Mat}_n(\K[\sZ^{\pm1}])$ is the adjacency matrix of $\Ga$.
\begin{definition}
We say that $\Ga$ is of  \emph{maximal dimension} if $\dim\cA_\Ga=n$.
\end{definition}
Let $\Ga$ be a graph of maximal dimension and assume that  there exists a matrix realisation $(\cA,\mB)$ of $\Ga$. Then since $\cA$ contains $A_\Ga$ and is of dimension $n$, we must have $\cA=\cA_\Ga$. In particular, $\cA$ is commutative. Note also that, in this case, the centralizer of $A_\Ga$ is
$\cA_\Ga$.
\smallskip
\begin{example}
\begin{enumerate}
\item The graph in Example \ref{mutl_graph_1} is not of maximal dimension. Indeed, if it was, the algebra $\K[\mathfrak{S}_n]$ would be isomorphic to a commutative algebra. 
\item The graphs $\Ga$ constructed from a polynomial $P$ as in Example \ref{mutl_graph_2} are of maximal dimension since the minimal polynomial of $A_\Ga$ is $P$ and the degree of $P$ is the number of vertices in $\Ga$.
\item Let $\mathcal{A}$ be the complex character algebra of a finite
group $G$ with positively multiplicative basis the set $\mB=\{\chi_{1},\ldots,\chi_{n}\}$
of irreducible characters. Let 
\[
\varphi=a_{1}\chi_{1}+\cdots+a_{n}\chi_{n}\qu{with $(a_{1},\ldots,a_{n})\in\mathbb{R}^{n}.$}
\]
Let $\Ga$ be the graph with $n$ vertices such that $A_\Ga =\func{Mat}_{\mB}(m_\varphi)$. In other words, the vertices of $\Ga$ are in bijection with the set of irreducible characters and the weights encode the tensor product multiplicities.
The basis~$\mB'=\{1_{C_{1}},\ldots,1_{C_{n}}\}$ of characteristic
functions associated to the conjugacy classes of $G$ is another positively multiplicative basis of
$\mathcal{A}$. Using $\mB'$, we see that $\cA$ is isomorphic to
the algebra~$\mathbb{C}^{n}$. It follows that $\Ga$ is of maximal
dimension if and only if $\varphi$ takes distinct values on each conjugacy class. Indeed the element $x=(x_1,\ldots,x_n)$ of $\mathbb{C}^{n}$ generates $\mathbb{C}^{n}$ if and only if its minimal polynomial has degree $n$. Since this polynomial coincides with that of the multiplication by $x$ in $\mathbb{C}^{n}$ whose eigenvalues are the coordinates $x_i,i=1,\ldots,n$, this is equivalent to say that the $x_i$'s are pairwise distinct.
\end{enumerate}
\end{example}


We first show that if $\Ga$ is of maximal dimension then it admits a matrix realisation. 
\begin{proposition}
\label{Prop_Unicity} Assume that $\Ga$ is of maximal dimension. The set 
$$\fB = \{\mB\mid (\cA_\Ga,\mB) \text{ is a matrix realisation of $\Ga$}\}$$
is non-empty and the invertible elements $\cA_\Ga^\times$ of $\cA_\Ga$ acts transitively by left multiplication on this set.
\end{proposition}

\begin{proof} We prove that the set $\fB$ is non empty. Since~$\dim \cA_\Ga=n$, the set $\mB^{\prime
}=\{A^{i},i=0,\ldots,n-1\}$ is a basis of $\cA_\Ga$ and $\func{Mat}_{\mB'}(m_{A_\Ga}) = \cC_{\mu_{A_\Ga}}$ where $\mu_{A_\Ga}$ is the minimal polynomial of $A_\Ga$ and $\cC_{\mu_{A_\Ga}}$ is the companion matrix of $\mu_{A_\Ga}$. The matrices $\cC_{A_\Ga}$ and $A_\Ga$ have entries in $\K(\sZ^{\pm1})$ and are conjugate in $\func{Mat}_n(\K(\sZ^{\pm1}))$ since $\dim\K(\sZ^{\pm1})[A_\Ga]=n$ is the degree of $\mu_{A_\Ga}$. Thus $A_\Ga=P\mathcal{C}_{\mu_{A}}P^{-1}$ where $P$ is an invertible matrix $P$ with entries in $\K(\sZ^{\pm1})$.  If we define $\mB$ to be the basis of $
\cA_\Ga$ such that the change of basis matrix from $\mB'$ to $\mB$ is equal to $P$, we then have  $\func{Mat}_{\mB} (m_{A_\Ga})= A_\Ga$ as required. 

\smallskip

We prove that $\cA_\Ga^\times$ acts by multiplication on $\fB$. Let $\mB = \{\mb_1,\ldots,\mb_n\}$ and $\mB'= \{\mb'_1,\ldots,\mb'_n\}$ be two bases in $\fB$. Let~$Q = (q_{i,j})_{1\leq i,j\leq n}$ be the change of basis matrix from $\mB'$ to $\mB$, that is $\mb'_j=\sum_i q_{i,j} \mb_j$. We have 
$\func{Mat}_{\mB'}(m_{A_\Ga}) =Q^{-1} \func{Mat}_{\mB}(m_{A_\Ga}) Q$
so that $A_\Ga = Q^{-1}A_\Ga Q$ and $A_\Ga$ commutes with $Q$. But $\Ga$ is of maximal dimension so the centralizer of $A_\Ga$ is 
equal to $\cA_\Ga$, therefore $Q\in \cA_\Ga$. Let $U\in \K(\sZ^{\pm 1})[X]$ be such that $Q=U(A_\Ga)$. Since $\func{Mat}_{\mB}(m_{A_\Ga})= A_\Ga$, we get that   $\func{Mat}_{\mB}(m_{Q})= Q$. But by definition, we have 
$$\mb_j' = \sum_{i=1}^n q_{i,j} \mb_j = Q\mb_j.$$
This shows that $\mB$ can be obtained from $\mB'$ by multiplying by the invertible element $Q$.

\smallskip

Finally, it is clear that if $\mB\in \fB$ then $x\mB\in \fB$ for all invertible elements $x\in \cA_\Ga^\times$. Therefore, we have a transitive action of the abelian group $\cA_\Ga^\times$ on $\fB$.
\end{proof}
We keep the notation of the proposition and we fix $i\in \{1,\ldots,n\}$.   The matrices in $\{\mB\croc{i}\mid \mB\in \fB\}$  
 are either all invertible or all not invertible. When they are all invertible, there exists a unique basis $\mB$ such that $(\cA_\Ga,\mB)$ is a matrix realisation of~$\Ga$ and $\mB\croc{i}=I_n$. The graph $\Ga$ is then multiplicative at $v_i$. When they are not invertible, the graph $\Ga$ is not multiplicative at $v_i$. In other word the previous proposition implies the following result. 
 \begin{corollary}
 \label{unique_normalised_basis}
Assume that $\Ga$ is of maximal dimension. If $\Ga$ is multiplicative at $v_{i_0}$ then there exists a unique matrix representation $(\cA,\mB)$ such that $\mB\croc{i_0}=I_n$.
 \end{corollary}

\begin{remark} 
 In the proof of the previous proposition, we have used the fact  that $\func{Mat}_{\mB}(m_{A_\Ga}) = A_\Ga$ implies that $\func{Mat}_{\mB}(m_{U(A_\Ga)}) = U(A_\Ga)$ for each polynomial $U\in  \K(\sZ^{\pm 1})[X]$. In general, any $n$-dimensional algebra embeds in the algebra of its linear endomorphisms just by considering the multiplication by each element (which is a linear map). This embedding is not surjective (the algebras have dimensions $n$ and $n^{2}$). We thus warn the reader that in general, the action of a linear
endomorphism of $\mathcal{A}$ does not coincide with the multiplication by its
matrix in the basis $\mB$ (which does not necessarily commute with
$\mathcal{A})$. 
\ This is in particular the case when we consider generalised
permutation matrices associated to the basis $\mB$.
%
\end{remark}

The situation simplifies even further when the minimal polynomial $\mu_{A_\Ga}$ of the adjacency matrix of $\Ga$ is irreducible and the graph $\Ga$ is strongly connected. 
\begin{proposition}
\label{Prop_unicity_irredu}
Assume that $\Ga$ is a strongly connected graph and that $\mu_{A_\Ga}$ is irreducible over
$\K(\sZ^{\pm1})$. Then $\Gamma$ is of maximal dimension and is multiplicative at any vertex. 
\end{proposition}

\begin{proof}
Consider the Frobenius reduction of the matrix $A_\Ga$ in the field $\K(\sZ^{\pm1})$.
Denote its invariant factors by $\mu_{1}=\mu_{A_\Ga},\mu_{2},\ldots,\mu_{r}$ with
$\mu_{i+1}\mid\mu_{i}$ for any $i=1,\ldots,r-1$. Since $\mu_{1}$ is
irreducible we must have $\mu_{1}=\mu_{2}=\cdots=\mu_{r}$.\ The matrix $A_\Ga$ is
conjugate in $\K(\sZ^{\pm1})$ to a matrix $B$ with $r$ identical blocks equal to
$\mathcal{C}_{\mu_{A_\Ga}}$ the companion matrix of $\mu_{A_\Ga}$. Let us write
$B=PAP^{-1}$. The matrix $P$ has entries in  $\K(\sZ^{\pm1})$ so it can be
written under the form $P=\frac{1}{d}P'$ where~$P'$ has
coefficients in~$\K[\sZ^{\pm1}]$ and $d\in\K[\sZ^{\pm 1}]$. As a consequence, there
exists a specialisation $f_{\theta}$ such that~$f_{\theta}(z)>0$ for any
$z\in  \sZ$ and $f_{\theta}(d)\neq0$. The matrix~$A_{\theta}$ (the specialisation of $A_\Ga$ via $\theta$) remains irreducible because $\Gamma$ is strongly connected,
thus its Perron-Frobenius eigenvalue~$\lambda$ corresponds to an eigenspace of
dimension $1$.\ This implies that $r=1$.\ Indeed, $A_{\theta}$ is
conjugate in $\mathbb{C}$ to the matrix $B_{\theta}$  (the specialisation of $B$ via $\theta$)  with $r$ identical
blocks, each of them having the eigenvalue~$\lambda$ because it is a root of
$\theta(\mu_{A_\Ga})$. Since $r=1$, the minimal polynomial $\mu_{A_\Ga}$ has degree
$n$ and we can apply Proposition~\ref{Prop_Unicity} to get a basis
$\mB$ such that $(\cA_\Ga,\mB)$ is a matrix representation of $\Ga$. Further since 
$\K(\sZ^{\pm1})[A_\Ga]\simeq \K(\sZ^{\pm1})[X]/\langle\mu_{A_\Ga}\rangle$ we know that $\K(\sZ^{\pm1})[A_\Ga]$ is a field. All the elements of $\mB$ are non-zero (otherwise $\mB$ is not a basis), hence they are invertible in $\K(\sZ^{\pm1})[A_\Ga]$ and this implies that $\Ga$ is multiplicative at any vertex.
\end{proof}

Recall that the coefficient $m_{i,j}= m_{i_0\to i}^{j}$ is the sum
of the weights of all paths in $\Ga$ of length $j$ from the vertex $v_{i_{0}}$ to the vertex $v_{i}$ and that $M_{i_0} = (m_{i_0\to i}^{j})_{1\leq i,j\leq n}$. 

\smallskip

The matrix $M_{i_0}$ gives
a simple combinatorial criterion to decide whether or not the matrix $\mB\croc{i_{0}}$ is invertible in a matrix realisation $(\cA,\mB)$ of $\Ga$ (recall that this does not depend on the matrix realisation). Further, when  $\mb_{i_{0}}$ is invertible, the matrix $M_{i_0}$ allows one to compute the corresponding unique normalised basis $\mB$ such that $\mB\croc{i_0}=I_n$ (see Corollary \ref{unique_normalised_basis}).
\begin{theorem}
\label{ThGG1}
 The graph $\Gamma$ is of  maximal dimension and is multiplicative at $v_{i_0}$ if and only if the matrix $M_{i_{0}}$ is invertible. Moreover, in this case, the
columns of $M_{i_{0}}^{-1}$ give the vectors of the normalised basis
$\mB=\{\mb_{1},\ldots,\mb_{i_{0}}=1,\ldots
,\mb_{n}\}$ expressed
in the basis $\{1,A_\Ga,\ldots,A_\Ga^{n-1}\}$. In particular, the entries of the matrices in the basis $\mB$ belong to~$\frac{1}{\det M_{i_{0}}}\K[\sZ^{\pm1}]$.
\end{theorem}

\begin{proof}
Assume that $\Ga$ is of maximal dimension and is multiplicative at $v_{i_0}$. Then there exists a basis $\mB=\{\mb_{1},\ldots,\mb_{n}\}$ of $\cA_\Ga$ such that $(\cA_\Ga,\mB)$ is a matrix realisation of $\Ga$ and $\mb_{i_0}=I_n$. Using (\ref{Transfert}) we get
\begin{equation*}
A_\Ga^{j}\mb_{i_{0}}=A_\Ga^{j}=\sum_{i=1}^{n}m_{i_0\to i}^{j}\mb_{i}. \label{A in B}%
\end{equation*}
Since $\dim\cA_\Ga = n$, this equation shows that the matrix $M_{i_0}$ is the change of basis matrix from the basis $\{I_n,A_\Ga,\ldots,A_\Ga^{n-1}\}$ to the basis $\mB$.
 Therefore, it is invertible and the columns of $M_{i_{0}}^{-1}$ express the vectors $\mB_{i_{0}}$ in the basis $\{I_n,A_\Ga,\ldots,A_\Ga^{n-1}\}$.

Assume now that $M_{i_{0}}$ is invertible. Let $(x_{0},\ldots,x_{n-1})\in \K(\sZ^{\pm1})^{n}$ such that $\sum
_{j=0}^{n-1}x_{j}A_\Ga^{j}=0$. The $j$-th column $C_{j}$ of $M_{i_0}$
coincides with the $i_{0}$-th column of $A_\Ga^{j}$ thus we have $\sum_{j=0}^{n-1}%
x_{j}C_{j}=0$ and since $M_{i_0}$ is invertible, we must have $x_i=0$ for all $i$. It follows that  $\dim \cA_\Ga=n$ and $\Ga$ is of maximal dimension.  By Proposition \ref{Prop_Unicity}, there exists a basis $\mB$ such that $(\cA_\Ga,\mB)$ is a matrix realisation of $\Ga$. The family $\{A_\Ga^{j}\mb_{i_{0}}\mid j=0,\ldots,n-1\}$ is also a basis of $\cA_\Ga$ (the change of basis matrix from $\{A_\Ga^{j}\mb_{i_{0}}\mid j=0,\ldots,n-1\}$ to $\mB$ is the invertible matrix $M_{i_0}$). 
But this can only be true if $\mb_{i_{0}}$ is invertible. Therefore $\Ga$ is multiplicative at $v_{i_0}$. 
\end{proof}
 
 \begin{remark}
 \label{compute}
Theorem \ref{ThGG1} provides a simple procedure to (1) decide whether or not a graph $\Ga$ of maximal dimension is multiplicative at a given vertex and (2) compute the unique corresponding matrix realisation. For instance, we can compute explicitely all the structure constants of the homology ring of affine Grassmannians for the affine Weyl groups of type $G_2$ \cite{G2}.

\end{remark}

The corollary below follows easily from the previous theorem and Proposition \ref{Prop_unicity_irredu}.
\begin{corollary}
Assume that $\Ga$ is strongly connected and that $\mu_{A_\Ga}$ is irreducible. Then the matrix $M_{i_0}$ is invertible for all $i_0$. 
\end{corollary}

\begin{example}
\label{Counterexample} 
Consider the graph $\Gamma$ with adjacency matrix $A_\Gamma$ given as follows:

\medskip

\begin{minipage}{8cm}
$$
\begin{tikzpicture}[scale =1]
\tikzstyle{vertex}=[inner sep=2pt,minimum size=10pt,ellipse,draw]

\node[vertex] (a1) at (0,0) {\scalebox{.7}{$v_1$}};
\node[vertex] (a2) at (1,-1) {\scalebox{.7}{$v_2$}};
\node[vertex] (a3) at (-1,-1) {\scalebox{.7}{$v_3$}};

\draw[->] (a1) edge node  {} (a2);
\draw[->] (a1) edge node{} (a3);
\draw[<->] (a2) edge node{} (a3);

\draw[->]  (a2) edge[bend right = 30] node[right]{\scalebox{.7}{$z_2$}} (a1);
\draw[->]  (a3) edge[bend left = 30] node[left]{\scalebox{.7}{$z_1$}} (a1);
\end{tikzpicture}$$
\end{minipage}
\begin{minipage}{4cm}
$$
A_\Ga=\left(
\begin{array}
[c]{ccc}%
0 & z_{1} & z_{2}\\
1 & 0 & 1\\
1 & 1 & 0
\end{array}
\right)
.$$
\end{minipage}

\medskip

\noindent
The graph $\Ga$ is strongly connected and we have $\mu_{A}(X)=\left(  X+1\right)  \left(  X^{2}-X-z_{1}-z_{2}\right)$ thus $\Ga$ is of maximal dimension. We compute
\[
M_{1}=\left(
\begin{array}
[c]{ccc}%
1 & 0 & z_{1}+z_{2}\\
0 & 1 & 1\\
0 & 1 & 1
\end{array}
\right)  ,\ M_{2}=\left(
\begin{array}
[c]{ccc}%
0 & z_{1} & z_{2}\\
1 & 0 & z_{1}+1\\
0 & 1 & z_{1}%
\end{array}
\right)  ,\ M_{3}=\left(
\begin{array}
[c]{ccc}%
0 & z_{2} & z_{1}\\
0 & 1 & z_{2}\\
1 & 0 & z_{2}+1
\end{array}
\right)
\]
and%
\[
\det M_{1}=0,\ \det(M_{2})=z_{2}-z_{1}^{2},\ \det(M_{3})=z_{2}^{2}-z_{1}.
\]
It follows that $\Gamma$ is multiplicative at $v_{2}$ (respectively at $v_{3}$) if and only if
$z_{2}\neq z_{1}^{2}$ (respectively $z_{1}\neq z_{2}^{2}$). It is not multiplicative at $v_1$. 
Observe that when
$z_{1}=z_{2}=1$, $\Gamma$ is not multiplicative but, according to Proposition \ref{Prop_Unicity}, it admits a matrix realisation. 
\end{example}

\begin{example}
\label{Exam_kSchur3} Consider the graph $\Gamma$ with adjacency matrix $A_\Gamma$ given as follows:

\medskip

\begin{minipage}{8cm}
$$
\begin{tikzpicture}[scale =1]
\tikzstyle{vertex}=[inner sep=2pt,minimum size=10pt,ellipse,draw]

\node[vertex] (a0) at (0,0) {\scalebox{.7}{$v_1$}};
\node[vertex] (a1) at (0,-1) {\scalebox{.7}{$v_2$}};
\node[vertex] (a2) at (-1,-2) {\scalebox{.7}{$v_3$}};
\node[vertex] (a3) at (1,-2) {\scalebox{.7}{$v_4$}};
\node[vertex] (a4) at (0,-3) {\scalebox{.7}{$v_5$}};
\node[vertex] (a5) at (0,-4) {\scalebox{.7}{$v_6$}};

\draw[->] (a0) edge node  {} (a1);
\draw[->] (a1) edge node{} (a2);
\draw[->] (a1) edge node{} (a3);
\draw[->] (a2) edge node{} (a4);
\draw[->] (a3) edge node{} (a4);
\draw[->] (a4) edge node{} (a5);

\draw[->]  (a2) edge[bend left = 30] node[left]{\scalebox{.7}{$z_1$}} (a0);
\draw[->]  (a3) edge[bend left = -30] node[right]{\scalebox{.7}{$z_3$}} (a0);

\draw[->]  (a4) edge[bend left = 30] node[right,pos=.3]{\scalebox{.7}{$z_2$}} (a0);
\draw[->]  (a5) edge[bend right = 30] node[right,pos=.3]{\scalebox{.7}{$z_2$}} (a1);

\draw[->]  (a5) edge[bend left = 30] node[left]{\scalebox{.7}{$z_3$}} (a2);
\draw[->]  (a5) edge[bend left = -30] node[right]{\scalebox{.7}{$z_1$}} (a3);

\end{tikzpicture}$$
\end{minipage}
\begin{minipage}{4cm}
$$
A_\Ga=\left(
\begin{array}
[c]{cccccc}%
0 & 0 & z_{1} & z_{3} & z_{2} & 0\\
1 & 0 & 0 & 0 & 0 & z_{2}\\
0 & 1 & 0 & 0 & 0 & z_{3}\\
0 & 1 & 0 & 0 & 0 & z_{1}\\
0 & 0 & 1 & 1 & 0 & 0\\
0 & 0 & 0 & 0 & 1 & 0
\end{array}
\right)
.$$
\end{minipage}

\medskip

\noindent
We get $\mu_{A}(T)=T^{6}-2\left(  z_{1}+z_{3}\right)  T^{3}-4z_{2}%
T^{2}+\left(  z_{1}-z_{3}\right)  ^{2}$ which is irreducible. The matrices
$M_{1}$ and $M_{1}^{-1}$ are respectively%
\[
\left(
\begin{array}
[c]{cccccc}%
1 & 0 & 0 & z_{1}+z_{3} & 2z_{2} & 0\\
0 & 1 & 0 & 0 & z_{1}+z_{3} & 4z_{2}\\
0 & 0 & 1 & 0 & 0 & z_{1}+3z_{3}\\
0 & 0 & 1 & 0 & 0 & 3z_{1}+z_{3}\\
0 & 0 & 0 & 2 & 0 & 0\\
0 & 0 & 0 & 0 & 2 & 0
\end{array}
\right)  \text{\ and }\left(
\begin{array}
[c]{cccccc}%
1 & 0 & 0 & 0 & -\frac{z_{1}+z_{3}}{2} & -z_{2}\\
0 & 1 & \frac{2z_{2}}{z_{1}-z_{3}} & \frac{2z_{2}}{z_{3}-z_{1}} & 0 &
-\frac{z_{1}+z_{3}}{2}\\
0 & 0 & \frac{3z_{1}+z_{3}}{2z_{1}-2z_{3}} & \frac{z_{1}+3z_{3}}{2z_{3}%
-2z_{1}} & 0 & 0\\
0 & 0 & 0 & 0 & \frac{1}{2} & 0\\
0 & 0 & 0 & 0 & 0 & \frac{1}{2}\\
0 & 0 & \frac{-1}{2z_{1}-2z_{3}} & \frac{-1}{2z_{3}-2z_{1}} & 0 & 0
\end{array}
\right)  \allowbreak.
\]

\medskip

\noindent
The graph $\Gamma$ is multiplicative at $v_{1}$. One can check that all the entries in $\{\mb_{1}=I_n,\ldots,\mb_{6}\}$ belong to $\mathbb{Z}%
[z_{1},z_{2},z_{3}]$. For example, we have
\[
\mb_{3}=\frac{1}{z_{1}-z_{3}}\left(  2z_{2}A_\Ga+\frac{1}{2}(3z_{1}+z_{3}%
)A_\Ga^{2}-\frac{1}{2}A_\Ga^{5}\right)  =\left(
\begin{array}
[c]{cccccc}%
0 & z_{1} & z_{2} & 0 & 0 & z_{1}z_{3}\\
0 & 0 & z_{1} & 0 & z_{2} & 0\\
1 & 0 & 0 & 0 & 0 & 0\\
0 & 0 & 0 & 0 & z_{1} & z_{2}\\
0 & 1 & 0 & 0 & 0 & z_{1}\\
0 & 0 & 0 & 1 & 0 & 0
\end{array}
\right). \]
\end{example}

Each element in $\cA_\Ga$ can be regarded as a linear map of
$\K(\sZ^{\pm 1})^{n}=\oplus_{i=0}^{n-1}\K(\sZ^{\pm 1})e_{i}$ where $(e_1,\ldots,e_n)$ is the canonical basis of $\K(\sZ^{\pm 1})^n$. In this context, saying that $M_{i_0}$ is invertible is equivalent to saying that the vector $e_{i_0}$ is cyclic\footnote{Recall that a vector $v$ in
	$K^{n}$ is cyclic for $A$ when $\{A^{m}v\mid0\leq m<n\}$ is a basis of $K^{n}$. Equivalently, the linear map $K[A]\rightarrow K^{n}$ sending any $B$ on $Bv$ is injective.}
for $A_\Ga$. Indeed, for any $1\leq k\leq n$, the vector $A_\Ga^{k}e_{i_{0}}$ is given by the $i_{0}$-th
column of $A_\Ga^{k}$ which coincides with the $k$-th column of $M_{i_{0}}$. Thus
$M_{i_{0}}$ is invertible if and only if $e_{i_{0}}$ is cyclic.

\begin{theorem}
\label{ThGG2} 
Assume that $\Ga$ is strongly connected, multiplicative at $v_{i_0}$ and of maximal dimension. Let $\mB=\{\mb_{1},\ldots,\mb_{n}\}$ be the unique basis of $\cA_\Ga$ such that $\mb_{i_0}=I_n$ and let $c_{i,j}^k$ be the structure constants $\cA_\Ga$ with respect to $\mB$. We have
\begin{enumerate}
\item $\mb_{i}%
(e_{i_{0}})=e_{i}$ for any $i=1,\ldots,n$, that is the $i_{0}$-th column of
$\mb_{i}$ is equal to $e_{i}$.
\item The $j$-th column of the matrix $\mb_{i}$ is equal to $(c_{i,j}^k)_{k=1,\ldots,n}$.
\item For any $1\leq i,j\leq n$, the $j$-th column of $b_{i}$ is
equal to the $i$-th column of $b_{j}$.
\item The entries of the matrices in the
basis $\mB$ and the coefficients $c_{i,j}^{k}$ all belong to
$\frac{1}{\det M_{i_{0}}}\K[\sZ^{\pm1}]$.
\end{enumerate}
\end{theorem}

\begin{proof}
We prove (1). Define the matrix $C$ by the equations $Ce_{i}=\mb_{i}(e_{i_{0}})$ for all $i\in \{1,\ldots,n\}$. To prove that
$C$ belongs to $\cA_\Ga$, it suffices to show that $C$ commutes with $A_\Ga$
($\Ga$ is of maximal dimension therefore the centralizer of $A_\Ga$ is $\cA_\Ga$). Let $j\in\{1,\ldots,n\}$. On the one hand, we have
\[
CA_\Ga e_{j}=C\left(  \sum_{i=1}^{n}a_{i,j}e_{i}\right)  =\sum_{i=0}^{n-1}%
a_{i,j}Ce_{i}=\sum_{i=0}^{n-1}a_{i,j}\mb_{i}e_{i_{0}}.
\]
On the other hand%
\[
A_{\Ga}Ce_{j}=(A_\Ga \mb_{j})e_{i_{0}}=\left(  \sum_{i=0}^{n-1}a_{i,j}\mb_{i}\right)
e_{i_{0}}=\sum_{i=0}^{n-1}a_{i,j}\mb_{i}e_{i_{0}}%
\]
where we use the fact that $\func{Mat}_{\mB}(m_{A_\Ga})= A_\Ga$  in the second equality. The matrix $C$ is
invertible because the set $\{\mb_{i}e_{i_{0}}\mid1\leq i\leq n\}$ is a basis of
$\K(\sZ^{\pm 1})^{n}$.\ Indeed, let $(x_{1},\ldots,x_{n})\in \K(\sZ^{\pm 1})^n$ be such that $\sum_{1\leq i\leq n}x_{i}\mb_{i}e_{i_0}=0$. Then $\sum_{1\leq i\leq n}x_{i}\mb_{i}=0$ because~$e_{i_{0}}$ is a cyclic vector for $A_\Ga$ but this implies that $x_i=0$ for all $i$. This
shows that $(\cA_\Ga,\mB')$ where $\mB'=C^{-1}\mB=(\mb'_1,\ldots,\mb'_n)$ is a matrix realisation of $\Ga$. Moreover we have \[
\mb'_{i}e_{i_{0}}=C^{-1}\mb_{i}e_{i_{0}}=C^{-1}Ce_{i}=e_{i}\qu{for all $i\in \{1,\ldots,n\}$.}
\]
It follows that $\mb'_{i}\mb'_{i_{0}}(e_{i_{0}})=\mb'_{i}(e_{i_{0}})=e_{i}$ 
for all $i\in \{1,\ldots,n\}$ and that $\mb'_{i}\mb'_{i_{0}} = \mb'_{i}$ since $e_{i_0}$ is cyclic for $A_\Ga$. Since $\mB'$ is a basis
of $\cA_\Ga$, this
implies that $\mb'_{i_{0}}=I_n$. By Corollary \ref{unique_normalised_basis}, we get $\mB=\mB'$ and  $C=I_n$ as desired.

\smallskip


We prove (2). We have
\[
\mb_{i}\mb_{j}e_{i_{0}}=\sum_{k=1}^{n}c_{i,j}^{k}\mb_{k}e_{i_{0}}%
\Longleftrightarrow \mb_{i}e_{j}=\sum_{k=0}^{n-1}c_{i,j}^{k}e_{k}%
\]
hence the result.

\smallskip 

Assertion (3) follows from the fact that $\mb_{i}\mb_{j}=\mb_{j}\mb_{i}$ in the commutative
algebra $\cA_\Ga$.

\smallskip 

We prove (4). Since $\Gamma$ is multiplicative at $v_{i_{0}}$ and is of maximal dimension, Theorem \ref{ThGG1} implies that the matrix $M_{i_0}$ is invertible and that the columns of $M_{i_0}$ belong to $\frac{1}{\det M_{i_{0}}}\K[\sZ^{\pm1}]$. Assertion (2) implies that the coefficients $c_{i,j}^k$ also belong to $\frac{1}{\det M_{i_{0}}}\K[\sZ^{\pm1}]$. 
\end{proof}

\section{Positively multiplicative graph}

\label{subset_rootmove}

In this section, $\Ga$ is a graph in $\sG{n}{\R_+[\sZ^{\pm 1}]}$ with set of vertices $\{v_1,\ldots,v_n\}$.

\subsection{Definitions and examples}
\begin{definition}
\label{Def_PMGraphs}
\begin{enumerate}
\item We say that $\Ga$ is \emph{positively multiplicative} at $v_{i_0}$  if there exists a matrix realisation $(\cA,\mB)$ of $\Ga$ such that $\cA$ is a PM algebra with respect to $\mB$ and $\mB\croc{i_0}=I_n$. We then say $\Ga$ is PM at $v_{i_0}$ with respect to the matrix realisation $(\cA,\mB)$.
\item We say that $\Ga$ is PM if there exists $v_{i_0}$ such that $\Ga$ is PM at $v_{i_0}$.  
\item The set $\cR^+_\Ga$ of \emph{positive roots} of $\Ga$ is the set of vertex $v_{i_0}$ such that $\Ga$ is PM  at~$v_{i_0}$.
\end{enumerate}
\end{definition}

\begin{remark}
Assume that $\Ga$ is PM with respect to the matrix realisation $(\cA,\mB)$. Then the algebra $\cA$ can be defined over the ring $\K[\sZ^{\pm 1}]$ and we have 
$$\cA = \bigoplus \K[\sZ^{\pm 1}] \mb_i \qu{where $\mB=\{\mb_1,\ldots,\mb_n\}$.}$$
Therefore, we can  specialise the indeterminates in $\sZ$ to any positive real numbers.
\end{remark}

Given a graph $\Ga$, it is not easy in general to decide whether or not it is (positively) multiplicative. In the case where $\Ga$ is of maximal dimension, one can use Theorem  \ref{ThGG1}  to decide whether or not the  graph is multiplicative and, when it is, to compute the unique corresponding matrix realisation $(\cA,\mB)$. Then, according to Theorem \ref{ThGG2}, the graph $\Ga$ will be positively multiplicative if and only if all the coefficients appearing in the matrices in basis $\mB$ are in $\R_+[\sZ^{\pm 1}]$. 

\smallskip

It is however fairly easy to contruct PM graph from PM algebras. Indeed, let $\mathcal{A}$ be a finite dimensional PM algebra with PM basis $\mB=\{1=\mb_{1},\mb_{2},\ldots,\mb_{n}\}$ and let 
\begin{equation}
s=\sum_{i=1}^{n}\beta_{i}\mb_{i}\text{ where }\beta=(\beta_{1},\ldots
,\beta_{n})\in\mathbb{R}_{+}^{n}. \label{generator_s}%
\end{equation}
Let $\Gamma$ be the graph with set of vertices $\{v_{1},\ldots,v_{n}\}$ defined by the equation $A_\Ga = \func{Mat}_{\mB}(m_s)$ (recall here that $m_s$ is the left multiplication by $s$). In general the graph $\Gamma$ is not strongly connected
or even connected.\ To insure the strong connectivity, we need additional
hypotheses on the algebra $\mathcal{A}$. For example, if $\mathcal{A}$ is
strongly positive and $\beta\in(\mathbb{R}_{+}^\ast)^{n}$, the graph $\Gamma$ is
strongly connected and coincides in fact with a complete weighted graph. 

\smallskip

We now give a series of example of positively multiplicative graphs.

\begin{example} In this example, we show that the set of roots $\cR_\Ga$ and the set of positive roots $\cR_\Ga^+$ can be different. Consider the graph $\Gamma$ with adjacency matrix $A_\Gamma$ given as follows where $q>0$:
$$
\begin{minipage}{8cm}
\begin{center}
\begin{tikzpicture}[scale =1.3]
\def\rx{.866}
\def\ry{1}
\tikzstyle{vertex}=[inner sep=2pt,minimum size=10pt,circle,draw]
\node[vertex] (a1) at (0,0){$v_1$};
\node[vertex] (a2) at (-\rx,-\ry){$v_2$};
\node[vertex] (a3) at (\rx,-\ry){$v_3$};

\draw[line width=.5pt,->] (a1) to (a2);
\draw[line width=.5pt,->] (a2) to (a3);
\draw[line width=.5pt,->] (a3) to (a1);
\draw[line width=.5pt,->,bend left= 30] (a3) to node[below]{{\footnotesize $q$}} (a2);

\end{tikzpicture}
\end{center}
\end{minipage}
\begin{minipage}{4cm}
\[
A_\Ga=\left(
\begin{array}
[c]{ccc}%
0 & 0 & 1\\
1 & 0 & q\\
0 & 1 & 0
\end{array}
\right)
\]
\end{minipage}
$$
We compute 
$$M_1=\begin{pmatrix}
1&0&0\\
0&1&0\\
0&0&1
\end{pmatrix},\quad
M_2=\begin{pmatrix}
0&0&1\\
1&0&q\\
0&1&0
\end{pmatrix}\qu{and}
M_3=\begin{pmatrix}
0&1&0\\
0&q&1\\
1&0&q
\end{pmatrix}.$$
All these matrices are invertible, hence showing that $\cR_\Ga = \{v_1,v_2,v_3\}$ by Theorem \ref{ThGG1}. Further, for $i=1,2,3$, there exists a unique matrix realisation $(\cA_\Ga,\mB_i)$ such that $\mB_i\croc{i} =I_3$. Again by Theorem \ref{ThGG1}, we get
$$\begin{array}{llllcc} 
\mB_1 &= \left\{ 
I_3,A_\Ga,A_\Ga^2
\right\},\vspace{.2cm}\\
\mB_2 &= \left\{ 
-qI_3+A_\Ga^2,I_3,A_\Ga
\right\},\vspace{.2cm}\\
\mB_3 &= \left\{ 
q^2+A_\Ga-qA_\Ga^2,-q+A_\Ga^2,I_3
\right\}.\\
\end{array}
$$
Both matrices $-qI_3+A_\Ga^2$ and $-q+A_\Ga^2$ have negative coefficients, hence by Theorem \ref{ThGG2}, the graph~$\Ga$ is not PM at $v_2$ and $v_3$. It is PM at $v_1$, hence showing that $\cR_\Ga^+=\{v_1\}$.
\end{example}
%
%

\begin{example}
\label{exa1}
In Example \ref{Exam_kSchur3}, one can check that the matrices of the
elements in the given basis have coefficients in $\R
_{+}[z_{1},z_{2},z_{3}]$. Thus the graph is positively multiplicative at
$v_{1}$. 
\end{example}

\begin{example}
We give an example of a positively multiplicative graph which is not of
maximal dimension. Recall that the characters of the symmetric group
$\mathfrak{S}_{n}$ are parametrised by the partitions $\lambda$ of $n$. Let
$\chi_{\lambda}$ be the character associated to $\lambda$. It was proved by Gamba and Radicati \cite{GamRad} (see also \cite[Sec. 7.13]{Ham} for a simpler proof) that%
\[
\chi_{\lambda}\times\chi_{(n-1,1})=(l_{\lambda}-1)\chi_{\lambda}+\sum_{\mu
\neq\lambda}\chi_{\mu}%
\]
where the sum runs over the partitions $\mu$ obtained by moving a box in the
Young diagram of $\lambda$ and $l_{\lambda}$ is the number of distinct parts
in $\lambda$. The algebra $\mathcal{A}$ of characters for $\mathfrak{S}_{n}$
is PM for the basis $\mB=\{\chi_{\lambda}\mid n\vdash \lambda\}$. This
implies that the graph $\Gamma$ whose adjacency matrix is that of
the multiplication by $\chi_{(n-1,1)}$ in the basis~$\mB$, called the Hamermesh graph, is PM. The
graph $\Gamma$ is not of maximal dimension in general. For example, for $n=4$,
we get the following graph and adjacency matrix 
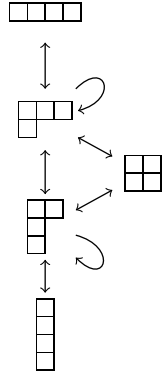
$$
\begin{minipage}{8cm}
\begin{center}
\begin{tikzpicture}[scale =1.3]
\def\rx{1.1}
\def\ry{0.866}
\tikzstyle{vertex}=[inner sep=2pt,minimum size=.8cm]
\node[vertex] (a1) at (0,0){$\scalebox{.4}{\ydiagram{4}}$};
\node[vertex] (a2) at (0,-\rx){$\scalebox{.4}{\ydiagram{3,1}}$};
\node[vertex] (a23) at (1,-1.5*\rx){$\scalebox{.4}{\ydiagram{2,2}}$};
\node[vertex] (a3) at (0,-2*\rx){$\scalebox{.4}{\ydiagram{2,1,1}}$};
\node[vertex] (a4) at (0,-3*\rx){$\scalebox{.4}{\ydiagram{1,1,1,1}}$};

\draw[line width=.5pt,<->] (a1) to (a2);
\draw[line width=.5pt,<->] (a2) to (a3);
\draw[line width=.5pt,<->] (a3) to (a4);
\draw[line width=.5pt,->] (a2) to[out=45,in=15,looseness=5]   (a2);
\draw[line width=.5pt,->] (a3) to[out=-15,in=-45,looseness=5]   (a3);
\draw[line width=.5pt,<->] (a2) to (a23);
\draw[line width=.5pt,<->] (a3) to (a23);

\end{tikzpicture}
\end{center}
\end{minipage}
\begin{minipage}{4cm}
\[
A=\left(
\begin{array}
[c]{ccccc}%
0 & 1 & 0 & 0 & 0\\
1 & 1 & 1 & 1 & 0\\
0 & 1 & 0 & 1 & 0\\
0 & 1 & 1 & 1 & 1\\
0 & 0 & 0 & 1 & 0
\end{array}
\right)
\]
\end{minipage}
$$
whose minimal polynomial has degree $4$.
\end{example}

In general, if we define a graph $\Gamma$ as in the example above starting from the algebra
$\mathcal{A}$ of complex characters of a finite group $G$ and its basis
$\mB$ of irreducible characters, it is known that the graph $\Gamma$
will be strongly connected if and only if the character $s$ in
(\ref{generator_s}) is the character of a faithful representation.


\subsection{Relabelling of vertices and weight changes on PM graphs}
\label{Subsec_Relabelling}
Assume that $\Ga$ is positively multiplicative graph at
$v_{i_{0}}$ with respect to the matrix realisation $(\cA,\mB)$. It is natural to
consider a larger class of graphs constructed from $\Gamma$ by authorizing the two
following variations (we will see that they preserve the property of being
positively multiplicative):

\begin{enumerate}
\item for a permutation $\sigma\in\mathfrak{S}_{n}$, switch the labels
of the vertices in $\Ga$: the label $i$ becomes~$\sigma(i)$.\ This gives a
new adjacency matrix $A^{\sigma}=T_{\sigma}^{-1}AT_{\sigma}$ where $T_{\sigma
}$ is the permutation matrix associated to $\sigma$, that is the $(i,j)$-coefficient of $T_\si$ is $\delta_{i,\sigma(j)}$.

\item for $\lambda_{1},\ldots,\lambda_{n}$ positive monomials in
$\mathbb{R}_{+}[\sZ^{\pm1}]$ and for any vertex $v_{j}$ in $\Gamma$, change each
edge $v_{j}\overset{a_{i,j}}{\rightarrow}v_{i}$ into $v_{j}\overset
{(\lambda_{j}/\lambda_{i})a_{i,j}}{\rightarrow}v_{i}$. This gives a new
adjacency matrix $A^{D}=D^{-1}AD$ where $D=\mathrm{diag}(\lambda_{1}%
,\ldots,\lambda_{n})$.
\end{enumerate}

\noindent We will denote by $\mathcal{G}_{n}$ the group of generalised permutation matrices of size $n$. These are matrices in which each row and each column
contains exactly one nonzero entry equal to a monomial in $
\R_+[\sZ^{\pm 1}]$.  
Any generalised permutation matrix of
$\mathcal{G}_{n}$ can be written under the form $P=DT_{\sigma}$. 
There is a natural action of~$\mathcal{G}_{n}$ on the set $\G{n}$: we define the graph $\Ga^P$ by the equation $A_{\Ga^P} = P^{-1}A_\Ga P$.   We are going to see that this action of $\mathcal{G}_{n}$ (by
conjugation on the adjacency matrices) restricts to the set of positively
multiplicative graphs. 

First, we have for all $j=1,\ldots,n$
\[
A\mb_{j}=\sum_{i=1}^{n}a_{i,j}\mb_{i}\Longleftrightarrow A\mb_{\sigma(j)}%
=\sum_{i=1}^{n}a_{\sigma(i),\sigma(j)}\mb_{\sigma(i)}\Longleftrightarrow
A^{\sigma}\mb_{\sigma(j)}'=\sum_{i=1}^{n}a_{i,j}^{\sigma}\mb_{\sigma
(i)}'%
\]
where $\mb'_{i}=T_{\sigma}^{-1}\mb_{i}T_{\sigma}$ for $i=1,\ldots,n$ and $A^\sigma=(a_{i,j}^{\sigma})$. Thus, by setting $\mb_{i}^{\sigma}=\mb'_{\sigma(i)}$ we obtain that the graph $\Gamma^{\sigma}$ with adjacency matrix
$A^{\sigma}$ is positively multiplicative at $v_{\sigma^{-1}(i_{0})}$ with respect to $(\cA,\mB_\si)$ where $\mB_\si=\{\mb_{i}^{\sigma}\mid
i=1,\ldots,n\}$.

Similarly, we get $A^{D}=(a_{i,j}\times\frac{\lambda_{j}}{\lambda_{i}})$ and
the equivalences%
\[
A\mb_{j}=\sum_{i=1}^{n}a_{i,j}\mb_{i}\Longleftrightarrow A\left(  \lambda
_{j}\mb_{j}\right)  =\sum_{i=1}^{n}\frac{\lambda_{j}}{\lambda_{i}}%
a_{i,j}\left(  \lambda_{i}\mb_{i}\right)  \Longleftrightarrow A^{D}\left(
\lambda_{j}\mb_{j}^{^{\prime\prime}}\right)  =\sum_{i=1}^{n}\frac{\lambda_{i}%
}{\lambda_{j}}a_{i,j}\left(  \lambda_{i}\mb_{i}^{^{\prime\prime}}\right)
\]
where $\mb_{i}^{^{\prime\prime}}=D^{-1}\mb_{i}D$ for any $i=0,\ldots,n$.\ Hence
the graph $\Gamma^{D}$ with adjacency matrix $A^{D}$ is positively
multiplicative at $v_{i_{0}}$ with respect to $(\cA,\mB_D)$ where $\mB_{D}
=\{\mb_{i}^{D}\mid i=1,\ldots,n\}$ and $\mb_{i}^{D}=\lambda_{i}D^{-1}\mb_{i}D$.\ 

\begin{proposition}
With the previous notation, for any $P=DT_{\sigma}\in\mathcal{G}_{n}$, the
graph $\Gamma$ is positively multiplicative at $v_{i_{0}}$ with respect to $(\cA,\mB)$ if and only if $\Gamma^{P}$ is positively
multiplicative at $v_{\sigma(i_{0})}$ with respect to $(\cA,\mB^P)$ where $\mB^{P}=\{\mb_{1}^{P},\ldots,\mb_{n}^{P}\}$ and $\mb_{i}^{P}=\lambda_{\sigma(i)}P^{-1}\mb_{\sigma(i)}P$ for all $i=1,\ldots,n$.
\end{proposition}


\begin{remark}
\label{Remark_G}When $P$ commutes with each element of the algebra
$\cA_\Ga$ (thus in particular with the adjacency matrix $A_\Ga$), we get
$\mb_{i}^{P}=\lambda_{\sigma(i)}\mb_{\sigma(i)}$ and thus $\mB^{P}$ is just the image of $\mB$ by the
generalised permutation matrix $P$.
\end{remark}

\section{Column and row Kirillov-Reshetikhin crystals of affine type A}

The goal of this section is to present two particular families of positively
multiplicative graphs that are closely connected with the representation theory of affine quantum groups (see \cite{BumpSchilling2017}). In contrast to their combinatorial definition which is
very simple, their associated multiplicative algebras have sophisticated
structure constants for which no simple combinatorial description is known.
They will provide us interesting illustrations of the notions and results
presented in the previous sections of the paper. 
\subsection{Partitions, column and row tableaux}

A partition of length $k$ is a sequence of integers~$\lambda=(\lambda_{1}%
\geq\cdots\geq\lambda_{k}\geq0)$. It is conveniently identified with its Young
diagram as illustrated below
\[
\lambda=(5,3,3,2)\longleftrightarrow%
\scalebox{.8}{\ydiagram{5,3,3,2}} \medspace.
\]
By a column tableau of length $k$ in the alphabet $\{1,2,\ldots,n\}$, we mean a filling of the Young diagram of shape $(1,1,\ldots,1)$ (with $k$ ones) by strictly increasing integers in $\{1,2,\ldots,n\}$  from top to bottom. We denote by $\col{k}{n}$ the set of such tableaux. By a row tableau of length $k$ in the alphabet $\{1,2,\ldots,n\}$, we mean a filling of the Young diagram of shape $(k)$  by increasing integers in $\{1,2,\ldots,n\}$  from left to right. We denote by $\row{k}{n}$ the set of such tableaux.

\begin{example}
The following two tableaux are column and row tableau, respectively
\[%
\scalebox{.8}{\begin{ytableau} 1 \\ 2\\3\end{ytableau}}
\qu{and}
\scalebox{.8}{
\begin{ytableau} 1&  2 &2 &4&5\end{ytableau}} \medspace.
\]
\end{example}

\subsection{The Kirillov-Reshetikhin (KR) column crystals}

Given integers $n\geq2$ and $1\leq k\leq n$, the KR crystal $B_{k,n}^{c}$ is
the  oriented graph with vertices the set of vertices $\col{k}{n}$ and edges as follows
\begin{enumerate}
\item[{\tiny $\bullet$}] there is an edge between $C$ and $C'$  if there exists $i\in\{1,\ldots,n\}$ such that  $i\in C$, $i+1\notin C'$ and
$C'$ is the column tableau obtained by replacing $i$ by $i+1$ in $C$;
\item[{\tiny $\bullet$}] there is an edge between $C$ and $C'$  if  $n\in C$, $1 \notin C'$ and
$C'$ is the column tableau obtained by replacing $n$ by $1$ in $C$ and by reordering the entries. 
\end{enumerate}
There is a simple bijection between $
\col{k}{n}$ and the set of partitions contained in the
rectangle $(n-k)^{k}$. It associates to each column tableau $C=\{c_{k}>c_{k-1}%
>\cdots>c_{1})$ the partition $\lambda=(c_{k}-k,c_{k-1}-k+1,\ldots,c_{1}-1\}$.

\begin{example}
\begin{enumerate}
\item Assume $n=5$ and $k=2$. The set of tableaux in  $\col{2}{5}$ is 
\[%
\begin{array}
[c]{cccccccccccc}
\scalebox{.8}{\begin{ytableau} 1 \\ 2\end{ytableau}}&
\scalebox{.8}{\begin{ytableau} 1 \\ 3\end{ytableau}}&
\scalebox{.8}{\begin{ytableau} 2 \\ 3\end{ytableau}}&
\scalebox{.8}{\begin{ytableau} 1\\ 4\end{ytableau}}&
\scalebox{.8}{\begin{ytableau} 2 \\ 4\end{ytableau}}&
\scalebox{.8}{\begin{ytableau} 1 \\ 5\end{ytableau}}&
\scalebox{.8}{\begin{ytableau} 3 \\ 4\end{ytableau}}&
\scalebox{.8}{\begin{ytableau} 2 \\ 5\end{ytableau}}&
\scalebox{.8}{\begin{ytableau} 3 \\ 5\end{ytableau}}&
\scalebox{.8}{\begin{ytableau} 4 \\ 5\end{ytableau}}&
\end{array}\medspace.
\]
They correspond to the partitions (included in the box $(3,3)$):
\[%
\begin{array}
[c]{cccccccccccc}
\emptyset &
\scalebox{.8}{\ydiagram{1}}&
\scalebox{.8}{\ydiagram{1,1}}&
\scalebox{.8}{\ydiagram{2}}&
\scalebox{.8}{\ydiagram{2,1}}&
\scalebox{.8}{\ydiagram{3}}&
\scalebox{.8}{\ydiagram{2,2}}&
\scalebox{.8}{\ydiagram{3,1}}&
\scalebox{.8}{\ydiagram{3,2}}&
\scalebox{.8}{\ydiagram{3,3}}&
\end{array}\medspace.
\]
The graph $B_{2,5}^c$ and the corresponding adjacency matrix with the ordering of $\col{2}{5}$ above are

\begin{minipage}{8cm}
\begin{center}
\begin{tikzpicture}[scale =1]
\tikzstyle{vertex}=[inner sep=2pt,minimum size=10pt]

\node[vertex] (a0) at (0,0) {$\scalebox{.5}{\begin{ytableau} 1 \\ 2\end{ytableau}}$};

\node[vertex] (a1) at (0,-1) {$\scalebox{.5}{\begin{ytableau} 1 \\ 3\end{ytableau}}$};

\node[vertex] (a2) at (-1,-2) {$\scalebox{.5}{\begin{ytableau} 2 \\ 3\end{ytableau}}$};
\node[vertex] (a3) at (1,-2) {$\scalebox{.5}{\begin{ytableau} 1 \\ 4\end{ytableau}}$};

\node[vertex] (a4) at (-1,-3) {$\scalebox{.5}{\begin{ytableau} 2 \\ 4\end{ytableau}}$};
\node[vertex] (a5) at (1,-3) {$\scalebox{.5}{\begin{ytableau} 1 \\ 5\end{ytableau}}$};

\node[vertex] (a6) at (-1,-4) {$\scalebox{.5}{\begin{ytableau} 3 \\ 4\end{ytableau}}$};
\node[vertex] (a7) at (1,-4) {$\scalebox{.5}{\begin{ytableau} 2 \\ 5\end{ytableau}}$};

\node[vertex] (a8) at (0,-5) {$\scalebox{.5}{\begin{ytableau} 3 \\ 5\end{ytableau}}$};

\node[vertex] (a9) at (0,-6) {$\scalebox{.5}{\begin{ytableau} 4 \\ 5\end{ytableau}}$};

\draw[line width=1pt,->] (a0)  --  (a1);
\draw[line width=1pt,->] (a1)  --  (a2);
\draw[line width=1pt,->] (a1)  --  (a3);
\draw[line width=1pt,->] (a2)  --  (a4);
\draw[line width=1pt,->] (a3)  --  (a4);
\draw[line width=1pt,->] (a3)  --  (a5);
\draw[line width=1pt,->] (a4)  --  (a6);
\draw[line width=1pt,->] (a4)  --  (a7);
\draw[line width=1pt,->] (a5)  --  (a7);

\draw[line width=1pt,->] (a6)  --  (a8);
\draw[line width=1pt,->] (a7)  --  (a8);

\draw[line width=1pt,->] (a8)  --  (a9);

\draw[->,line width=1pt,bend right= 50] (a7) to (a0);
\draw[->,line width=1pt] (a8) .. controls (-3,-4) and (-3,-2) .. (a1);
\draw[->,line width=1pt] (a9) .. controls (4,-4) and (2,-2) .. (a3);

\end{tikzpicture}
\end{center}
\end{minipage}
\begin{minipage}{4cm}
$$
A_{B^c_{2,5}}=\left(
\begin{array}
[c]{cccccccccc}%
0 & 0 & 0 & 0 & 0 & 0 & 1 & 0 & 0 & 0\\
1 & 0 & 0 & 0 & 0 & 0 & 0 & 0 & 1 & 0\\
0 & 1 & 0 & 0 & 0 & 0 & 0 & 0 & 0 & 0\\
0 & 1 & 0 & 0 & 0 & 0 & 0 & 0 & 0 & 1\\
0 & 0 & 0 & 1 & 0 & 0 & 0 & 0 & 0 & 0\\
0 & 0 & 1 & 1 & 0 & 0 & 0 & 0 & 0 & 0\\
0 & 0 & 0 & 0 & 1 & 1 & 0 & 0 & 0 & 0\\
0 & 0 & 0 & 0 & 0 & 1 & 0 & 0 & 0 & 0\\
0 & 0 & 0 & 0 & 0 & 0 & 1 & 1 & 0 & 0\\
0 & 0 & 0 & 0 & 0 & 0 & 0 & 0 & 1 & 0
\end{array}
\right)
$$
\end{minipage}

The  minimal polynomial of $B^c_{2,5}$ is $X^{10}-11X^{5}-1$ so it is of maximal dimension.

\item When $n=6$ and $k=2$, one can check that the minimal polynomial of $B_{2,6}^{c}$ is
$X^{13}-26X^{7}-27X$ so that the graph $B^a_{2,6}$ is not of maximal dimension. 
\end{enumerate}
\end{example}

\subsection{The KR row crystals}
Given integers $n\geq2$ and $1\leq \ell\leq n$, the KR crystal $B_{\ell,n}^{r}$ is
the  oriented graph with vertices the set of vertices $\row{\ell}{n}$ and edges as follows:
\begin{enumerate}
\item[{\tiny $\bullet$}] there is an edge between $L$ and $L'$  if there exists $i\in\{1,\ldots,n\}$ such that  $i\in L$ and
$L'$ is the row tableau obtained by replacing $i$ by $i+1$ in $L$;
\item[{\tiny $\bullet$}] there is an edge between $L$ and $L'$  if  $n\in L$ and
$L'$ is the column tableau obtained by replacing~$n$ by $1$ in $L$ and by reordering the entries. 
\end{enumerate}
There is a simple bijection between $
\row{k}{n}$ and the set of partitions contained in the
rectangle $(\ell,\ldots,\ell)$ ($n-1$ terms). It associates to each row tableau $L=\{c_{k}\geq c_{k-1}%
\geq \cdots\geq c_{1})$ the partition $\lambda$ containing a column of length $c_m-1$ for all $1\leq m\leq k$. In other words, the image of $\lambda$ is $(c_k-1,\ldots,c_1-1)$.
\begin{example}
The graph $B^s_{3,3}$ is as follows. On the left we have indexed the vertices with the set $\row{k}{n}$ and on the right by the corresponding partitions using the bijection described above. 

\medskip

\begin{center}
\begin{tikzpicture}[scale =1.7]
\def\rx{0.5}
\def\ry{0.75}
\tikzstyle{vertex}=[inner sep=2pt,minimum size=10pt]

\node[vertex] (a0) at (0,0) {$\scalebox{.7}{\begin{ytableau} 1&1&1\end{ytableau}}$};

\node[vertex] (a1) at (\rx,\ry) {$\scalebox{.7}{\begin{ytableau} 1&1&2\end{ytableau}}$};
\node[vertex] (a2) at (-\rx,\ry) {$\scalebox{.7}{\begin{ytableau} 1&1&3\end{ytableau}}$};

\node[vertex] (a11) at (2*\rx,2*\ry) {$\scalebox{.7}{\begin{ytableau} 1&2&2\end{ytableau}}$};
\node[vertex] (a12) at (0*\rx,2*\ry) {$\scalebox{.7}{\begin{ytableau} 1&2&3\end{ytableau}}$};
\node[vertex] (a22) at (-2*\rx,2*\ry) {$\scalebox{.7}{\begin{ytableau} 1&3&3\end{ytableau}}$};

\node[vertex] (a111) at (3*\rx,3*\ry) {$\scalebox{.7}{\begin{ytableau} 2&2&2\end{ytableau}}$};
\node[vertex] (a112) at (1*\rx,3*\ry) {$\scalebox{.7}{\begin{ytableau} 2&2&3\end{ytableau}}$};
\node[vertex] (a122) at (-1*\rx,3*\ry) {$\scalebox{.7}{\begin{ytableau} 2&3&3\end{ytableau}}$};
\node[vertex] (a222) at (-3*\rx,3*\ry) {$\scalebox{.7}{\begin{ytableau} 3&3&3\end{ytableau}}$};

\draw[line width=.5pt,->] (a0) to   (a1);
\draw[line width=.5pt,->] (a1) to   (a11);
\draw[line width=.5pt,->] (a2) to   (a12);
\draw[line width=.5pt,->] (a11) to   (a111);
\draw[line width=.5pt,->] (a12) to   (a112);
\draw[line width=.5pt,->] (a22) to   (a122);

\draw[line width=.5pt,->] (a1) to   (a2);
\draw[line width=.5pt,->] (a11) to   (a12);
\draw[line width=.5pt,->] (a12) to   (a22);
\draw[line width=.5pt,->] (a111) to   (a112);
\draw[line width=.5pt,->] (a112) to   (a122);
\draw[line width=.5pt,->] (a122) to   (a222);

\draw[line width=.5pt,->] (a22) to   (a2);
\draw[line width=.5pt,->] (a2) to   (a0);
\draw[line width=.5pt,->] (a12) to   (a1);
\draw[line width=.5pt,->] (a112) to   (a11);
\draw[line width=.5pt,->] (a122) to   (a12);
\draw[line width=.5pt,->] (a222) to   (a22);
\end{tikzpicture}
\hspace{1cm}
\begin{tikzpicture}[scale =1.7]
\def\rx{0.5}
\def\ry{0.75}
\tikzstyle{vertex}=[inner sep=2pt,minimum size=10pt]

\node[vertex] (a0) at (0,0) {$\scalebox{1}{$\emptyset$}$};

\node[vertex] (a1) at (\rx,\ry) {$\scalebox{.5}{\ydiagram{1}}$};
\node[vertex] (a2) at (-\rx,\ry) {$\scalebox{.5}{\ydiagram{1,1}}$};
\node[vertex] (a11) at (2*\rx,2*\ry)  {$\scalebox{.5}{\ydiagram{2}}$};
\node[vertex] (a12) at (0*\rx,2*\ry) {$\scalebox{.5}{\ydiagram{2,1}}$};
\node[vertex] (a22) at (-2*\rx,2*\ry)  {$\scalebox{.5}{\ydiagram{2,2}}$};

\node[vertex] (a111) at (3*\rx,3*\ry)  {$\scalebox{.5}{\ydiagram{3}}$};
\node[vertex] (a112) at (1*\rx,3*\ry) {$\scalebox{.5}{\ydiagram{3,1}}$};
\node[vertex] (a122) at (-1*\rx,3*\ry) {$\scalebox{.5}{\ydiagram{3,2}}$};
\node[vertex] (a222) at (-3*\rx,3*\ry) {$\scalebox{.5}{\ydiagram{3,3}}$};
\draw[line width=.5pt,->] (a0) to   (a1);
\draw[line width=.5pt,->] (a1) to   (a11);
\draw[line width=.5pt,->] (a2) to   (a12);
\draw[line width=.5pt,->] (a11) to   (a111);
\draw[line width=.5pt,->] (a12) to   (a112);
\draw[line width=.5pt,->] (a22) to   (a122);

\draw[line width=.5pt,->] (a1) to   (a2);
\draw[line width=.5pt,->] (a11) to   (a12);
\draw[line width=.5pt,->] (a12) to   (a22);
\draw[line width=.5pt,->] (a111) to   (a112);
\draw[line width=.5pt,->] (a112) to   (a122);
\draw[line width=.5pt,->] (a122) to   (a222);

\draw[line width=.5pt,->] (a22) to   (a2);
\draw[line width=.5pt,->] (a2) to   (a0);
\draw[line width=.5pt,->] (a12) to   (a1);
\draw[line width=.5pt,->] (a112) to   (a11);
\draw[line width=.5pt,->] (a122) to   (a12);
\draw[line width=.5pt,->] (a222) to   (a22);

\end{tikzpicture}
\end{center}

\end{example}

\subsection{Quotients of the algebra of symmetric polynomials}
\

Let $\Lambda_{k}=\text{Sym}[x_{1},\ldots,x_{k}]=\mathbb{Q}[e_{1},\ldots,e_{k}]$ be
the algebra of symmetric polynomials in the $k$ variables $x_{1},\ldots,x_{k}%
$.\ Here,
\[
e_{m}=\sum_{1\leq i_{1}<\cdots<i_{m}\leq k}x_{i_{1}}\cdots x_{i_{m}}\text{ for
any }1\leq m\leq k\text{.}%
\]
Write $\mathcal{P}_{k}$ for the set of partitions with at most $k$ parts. It
is well-known that $\Lambda_{k}$ admits the two distinguished bases%
\[
\{e_{\lambda'}\mid\lambda\in\mathcal{P}_{k}\}\text{ and }\{s_{\lambda
}\mid\lambda\in\mathcal{P}_{k}\}
\]
where for any partition $\lambda=(\lambda_{1},\ldots,\lambda_{k})$ we have
$e_{\lambda'}=e_{\lambda_{1}'}\cdots e_{\lambda_{k}'}$ and
$s_{\lambda}$ is the Schur polynomial associated to $\lambda$. Here $\lambda'$ is the conjugate partition of $\lambda$ whose parts are the heights of the columns in its Young diagram. For any
positive integer $\ell$, consider $\mathcal{I}_{k,\ell}$ the ideal of $\Lambda_{k}$ defined
by
\[
\mathcal{I}_{k,\ell}=\langle s_{\lambda}\mid\lambda\text{ has }\ell+1\text{ columns
of height less than }k\rangle.
\]
Write $\mathcal{P}_{k}^{\ell}$ for the set of partitions in $\mathcal{P}_{k}$ with at
most $\ell$ columns of height less than $k$. Recall that for any nonnegative integer $m$, the $m$-th homogeneous symmetric polynomial is defined by 
\[
h_{m}=\sum_{1\leq i_{1}\leq\cdots\leq i_{m}\leq k}x_{i_{1}}\cdots x_{i_{m}}.
\]
\begin{lemma}
\label{Lem_Ik,l}(See \cite{LM} Section 4) We have $I_{k,\ell}=\langle h_{\ell+1},\ldots,h_{\ell+k-1}\rangle$.
\end{lemma}
We define the ideals
\[
I_{k,\ell}^{q}=\langle h_{\ell+1},\ldots,h_{\ell+k-1},h_{\ell+k}-(-1)^{k-1}q\rangle\subset \mathbb{Q}[q]\otimes\Lambda_k \text{
and }J_{k,\ell}=\langle h_{\ell+1},\ldots,h_{\ell+k-1},e_{k}-1\rangle
\]
and set
\[
\mathcal{Q}_{k,\ell}:=\Lambda_{k}/\mathcal{I}_{k,\ell},\mathcal{A}_{k,\ell}%
^{q}:=\mathbb{Q}[q]\otimes\Lambda_k/I_{k,\ell}^{q}\text{ and }\mathcal{S}_{k,\ell}:=\Lambda
_{k}/J_{k,\ell}.
\]
Clearly the algebras $\mathcal{Q}_{k,\ell}$ and $\mathcal{A}_{k,\ell}^{q}$ are
isomorphic by sending $h_{\ell+k}\mod \mathcal{I}_{k,\ell}$ on
$(-1)^{k-1}q$. They were considered in \cite{BCF} and \cite{LM} from where we
can extract some of their important properties.

\begin{theorem}
\ \label{Th_Summ}(see \cite{LM}) Set $n=\ell+k$.

\begin{enumerate}
\item The set $\mB_{k,\ell}^{a}=\{b_{\lambda}:=s_{\lambda}%
\operatorname{mod}\mathcal{A}_{k,\ell}^{q}\mid\lambda\subset \ell^{k}\}$ is a
basis of $\mathcal{A}_{k,\ell}^{q}$ so that $\dim
\mathcal{A}_{k,\ell}^{q}=\binom{n}{k}$.

\item After specialising $q=1$, we have in $\mathcal{A}_{k,\ell}^{1}:=\mathcal{A}_{k,\ell}^{q}/\langle q=1\rangle$ for any $\lambda\subset
\ell^{k}$%
\[
b_{(1)}\cdot b_{\lambda}=\sum_{\nu\subset \ell^{k}}b_{\nu}%
\]
where $\nu$ is obtained by adding one box to $\lambda$ or by deleting the first
row and the first column when $\lambda_{1}=\ell$ and $\lambda$ has $k$ rows (i.e. by
removing the unique possible hook of length $\ell+k-1$ if any). In particular, in
the algebra $\mathcal{A}_{k,\ell}^{1}$, the matrix of the multiplication by
$b_{(1)}$ in the basis $b_{\lambda}$ is the adjacency matrix of the KR-crystal
$B_{k,n}$.

\item The structure constants of $\mathcal{A}_{k,\ell}^{1}$ with respect to $\mB_{k,\ell}^{a}$  are nonnegative integers.
\end{enumerate}
\end{theorem}

\begin{corollary}
The graph $B_{k,n}^{a}$ with $n=\ell+k$ is positively multiplicative.
\end{corollary}
\begin{proof}
 The previous theorem together with Proposition \ref{m_x} (taking $\cA = \cA_{k,\ell}^{1}$, $\mB= \mB_{k,\ell}^{a}$ and $x=b_{(1)}$) shows that the graph $B_{k,n}$ is multiplicative. Since the structure constants  of $\mathcal{A}_{k,\ell}^{1}$ with respect to $\mB_{k,\ell}^{a}$  are nonnegative integers, the result follows. 
\end{proof}

Now, let us turn to the algebra $\mathcal{S}_{k,\ell}:=\Lambda_{k}/J_{k,\ell}%
=\mathcal{Q}_{k,\ell}/\langle e_{k}=1\rangle$.

\begin{theorem}
(See \cite{Beauv,BCF})In the algebra $\mathcal{S}_{k,\ell}$ the following
statements hold.

\begin{enumerate}
\item The set $\mB_{k,\ell}^{s}=\{b_{\lambda}\operatorname{mod}%
J_{k,\ell}\mid\lambda\subset \ell^{k-1}\}$ is a basis of $\mathcal{S}%
_{k,\ell}$ so that $\dim\mathcal{S}_{k,\ell}=\binom{\ell+k-1}{\ell}$.

\item In $\mathcal{S}_{k,\ell}$, we have for any $\lambda\subset \ell^{k-1}$%
\[
b_{(1)}\cdot b_{\lambda}=\sum_{\nu\subset \ell^{k-1}}b_{\nu}%
\]
where $\nu$ is obtained by adding one box to $\lambda$, next by deleting a
column of height $k$ if such a column appears. In particular, in the algebra
$\mathcal{S}_{k,\ell}$, the matrix of the multiplication by $b_{1}$ in the basis
$\mB_{k,\ell}^{s}$ is the adjacency matrix of the KR-crystal labelled by
the rows of length $l$ on $\{1,\ldots,k\}$.

\item The structure constants associated of $\mathcal{S}_{k,\ell}$ with respect to $\mB_{k,\ell}^{s}$ are nonnegative integers.
\end{enumerate}
\end{theorem}

\begin{corollary}
The graph $B_{l,k}^{s}$ is positively multiplicative.
\end{corollary}
\begin{proof}
The previous theorem together with Proposition \ref{m_x} (taking $\cA = \mathcal{S}_{k,\ell}$, $\mB=\mB_{k,\ell}^{s}$ and $x=b_{(1)}$) shows that the graph $B_{l,k}^{s}$ is multiplicative. Since the structure constants  of $\mathcal{S}_{k,\ell}$ with respect to $\mB_{k,\ell}^{a}$  are nonnegative integers, the result follows. 

\end{proof}
\subsection{An example: the graph $B_{\ell,2}^{r}$}
\label{SubsecKR_Row2}
 We consider KR crystal $B_{\ell,2}^{s}$ with one row and $\ell$ columns filled with integers in $\{1,2\}$. The first examples of these graphs for $\ell=1,2$ and $3$ are :
\def\rx{1.5}
\def\ry{0.866}

\begin{center}
\begin{tikzpicture}[scale =1]

\tikzstyle{vertex}=[inner sep=2pt,minimum size=10pt]

\node[vertex] (a0) at (0,0) {$\scalebox{.7}{\begin{ytableau} 1\end{ytableau}}$};
\node[vertex] (a1) at (\rx,0) {$\scalebox{.7}{\begin{ytableau} 2\end{ytableau}}$};

\draw[line width=.5pt,->,bend left = 20] (a0) to   (a1);
\draw[line width=.5pt,->,bend left = 20] (a1) to   (a0);

\end{tikzpicture}
\end{center}

\medskip

\begin{center}
\begin{tikzpicture}[scale =1.4]

\tikzstyle{vertex}=[inner sep=2pt,minimum size=10pt]

\node[vertex] (a0) at (0,0) {$\scalebox{.7}{\begin{ytableau} 1&1\end{ytableau}}$};
\node[vertex] (a1) at (\rx,0) {$\scalebox{.7}{\begin{ytableau} 1&2\end{ytableau}}$};
\node[vertex] (a2) at (2*\rx,0) {$\scalebox{.7}{\begin{ytableau} 2&2\end{ytableau}}$};

\draw[line width=.5pt,->,bend left = 10] (a0) to   (a1);
\draw[line width=.5pt,->,bend left = 10] (a1) to   (a0);

\draw[line width=.5pt,->,bend left = 10] (a1) to   (a2);
\draw[line width=.5pt,->,bend left = 10] (a2) to   (a1);

\end{tikzpicture}
\end{center}

\medskip

\begin{center}
\begin{tikzpicture}[scale =1.4]

\tikzstyle{vertex}=[inner sep=2pt,minimum size=10pt]

\node[vertex] (a0) at (0,0) {$\scalebox{.7}{\begin{ytableau} 1&1&1\end{ytableau}}$};
\node[vertex] (a1) at (\rx,0) {$\scalebox{.7}{\begin{ytableau} 1&1&2\end{ytableau}}$};
\node[vertex] (a2) at (2*\rx,0) {$\scalebox{.7}{\begin{ytableau} 1&2&2\end{ytableau}}$};
\node[vertex] (a3) at (3*\rx,0) {$\scalebox{.7}{\begin{ytableau} 2&2&2\end{ytableau}}$};

\draw[line width=.5pt,->,bend left = 10] (a0) to   (a1);
\draw[line width=.5pt,->,bend left = 10] (a1) to   (a0);

\draw[line width=.5pt,->,bend left = 10] (a1) to   (a2);
\draw[line width=.5pt,->,bend left = 10] (a2) to   (a1);

\draw[line width=.5pt,->,bend left = 10] (a2) to   (a3);
\draw[line width=.5pt,->,bend left = 10] (a3) to   (a2);

\end{tikzpicture}
\end{center}
The adjacency matrix of $B_{\ell,2}^{s}$ is the $(\ell+1)\times
(\ell+1)$-matrix
$$A_{B_{\ell,2}^{s}} = 
\begin{pmatrix}
0&1&0&0&\ldots&0\\
1&0&1&0&\ldots&0\\
0&1&0&1&\ldots&0\\
0&0&\ddots&\ddots&\ddots&0\\
0&0&\ldots&1&0&1\\
0&0&\ldots&0&1&0
\end{pmatrix}.$$
\newcommand{\cL}{\mathcal{L}}

The algebra $\mathcal{S}_{2,\ell}$ is $\mathrm{Sym}[x_{1},x_{2}]/(e_{2}=1,h_{l+1}=0)$. Since
$e_{2}=x_{1}x_{2}$, we just have $x_{2}=x_{1}^{-1}$ and by setting $x=x_{1}$,
we see that the algebra $\mathcal{S}_{2,\ell}$ is equal to the
algebra $\cL[x^{\pm1}]$ of Laurent polynomials $P$ in the
indeterminate $x$ such that $P(x^{-1})=P(x)$. Now in $\cL[x^{\pm1}]$,
we have%
\[
h_{a}=\underset{i+j = a}{\sum_{1\leq i\leq j\leq \ell}}x^{i}x^{-j}=x^{-a}%
\sum_{i=0}^{a}x^{2i}=\frac{x^{a+1}-x^{-a-1}}{x-x^{-1}}.
\]
Therefore $\mathcal{S}_{2,\ell}$ is isomorphic to the algebra $\cL[x^{\pm1}]/\langle x^{\ell+2}-x^{-(\ell+2)}\rangle$. A simple computation shows that
\[
h_{i}\times h_{j}=\sum_{k=0}^{j}h_{i+j-2k}\qu{for all $i\geq j\in \N$.}
\]
Let $\overline{x}$ and $\mb_{i}$ be the image of $x$ and $h_i$ respectively in the quotient $\cL[x^{\pm1}]/\langle x^{\ell+2}-x^{-(\ell+2)}\rangle$ and define $\mB%
=\{\mb_{0}=1,\ldots,\mb_{\ell}\}$. We have $\overline{x}^{\ell+2}=\overline{x}^{-\ell-2}.$
Therefore, for any $0\leq a\leq \ell-1$, we get
\[
\mb_{\ell+1+a}=\frac{\overline{x}^{\ell+a+2}-\overline{x}^{-\ell-a-2}}{\overline
{x}-\overline{x}^{-1}}=\frac{\overline{x}^{-\ell-2+a}-\overline{x}^{\ell+2-a}%
}{\overline{x}-\overline{x}^{-1}}=-\mb_{\ell+1-a}.
\]
This gives assuming that  $i\geq j$:
$$\mb_{i}\mb_{j}=\mb_{i-j}+\mb_{i-j+2}+\cdots+\mb_{i+j}\text{ when }i+j\leq \ell$$
and
\begin{align*}
\mb_{i}\mb_{j}&=\sum_{k=0}^{j}\mb_{i+j-2k}=\mb_{i-j}+\cdots+\mb_{\ell}+0-\mb_{\ell}%
-\cdots-\mb_{i+j-\ell-1}\\
&=\mb_{i-j}+\cdots+\mb_{2(\ell+1)-i-j-1}\text{ when }i+j>\ell.
\end{align*}
This can be summarised by the rules%
\[
\mb_{i}\mb_{j}=%
\begin{array}
[c]{l}%
\mb_{\left\vert i-j\right\vert }+\cdots+\mb_{i+j}\text{ if }i+j\leq \ell,\\
\mb_{\left\vert i-j\right\vert }+\cdots+\mb_{2(\ell+1)-i-j-1}\text{ if }i+j>\ell.
\end{array}
\]
One may recognise the fusion rules for the fusion algebra $\widehat{su}_{2}$
(see \cite{DMS}) at level $\ell$. This construction is related to Conformal Field
Theory.\ The Perron-Frobenius eigenvalue of the adjacency matrix of $B_{\ell,2}^{c}$ is $\lambda=2\cos\frac{\pi}{l+2}$ with
normalised associated eigenvector $$\frac{1}{\sin\frac{\pi}{\ell+2}}(\sin
\frac{a\pi}{\ell+2})_{1\leq a\leq \ell+1}.$$

\section{Infinite PM graphs and harmonic functions}

The goal of this section is to present a general combinatorial construction
(called expansion) yielding infinite graphs from finite ones. To
simplify the exposition, we will also assume that all the algebras considered in this section
are commutative.

\smallskip

 Let $\Ga$ be a possibly infinite oriented graph with set of vertices $V$. We say that $\Ga$ is graded if there exists a partition $\{V_i\mid i\in \N\}$ of $V$ such that for any~$i\geq 1$, the edges in $\Ga$ which start at $V_i$ finish at $V_{i+1}$. In addition, when $V_0$ is a singleton, the graph $\Ga$ is said rooted and graded.

\smallskip

Recall that $\sZ=\{z_{1},\ldots,z_N\}$ with $N\in \N$ is a set (possibly empty) of formal indeterminates. For $\be=(\be_1,\ldots,\be_N)\in \Z^N$, we define $z^\beta = z_1^{\be_1}z_2^{\beta_2}\ldots z_{N}^{\beta_N}$. Then the set $\{z^\beta\mid \be\in\Z^N\}$ form a $\K$-basis of~$\K[\sZ^{\pm1}]$. 
Given $a\in \K[\sZ^{\pm1}]$, we denote by $a[\beta]$ the cofficient of $z^\beta$ in the expansion of $a$ in this basis. In other words, we have 
$$a = \sum_{\beta\in \Z^N} a[\beta] z^\be.$$
We write ${\bf 0} = (0,\ldots,0)\in \Z^N$.

\subsection{Expansion of a graph}
\label{Subsec_expansion}
In this section, $\Gamma\in \mathsf{Graph}_n(\R_+[\sZ^{\pm1}])$ denotes a finite strongly connected graph with set of vertices $V=\{v_1,\ldots,v_n\}$, edge weight function $\om$ and adjacency matrix $A_\Ga=(a_{i,j})$ in $\func{Mat}_{n}(\R_+[\sZ^{\pm 1}])$. Recall that $\om$ takes values in $\R_+[\sZ^{\pm 1}]$.

\smallskip


A
path $\pi$ of length $\ell$ on $\Gamma$ is a sequence of $\ell+1$ vertices of
$\Gamma$ with two consecutive vertices being connected by
an oriented edge. The
weight $\mathrm{wt}(\pi)$ of the path $\pi$ is the product of the weights of
the edges encountered. Therefore $\mathrm{wt}(\pi)$ is a monomial of
the form $cz^{\beta}$ with~$\beta\in \Z^N$ and $c\in \R_+$.


\begin{definition}
The expansion $\Ga_{{\sf e}}$ of $\Ga$ at $v_{i_0}$ is the rooted graded graph with vertices 
$$V_{\se}\subset \{{\sf v}^k_{i,z^\be}\mid i\in \{1,\ldots,n\},\beta\in \Z^N,\ell\in \N\}\qu{where ${\sf v}^k_{i,\be}=(v_i
,z^\beta,\ell)$}$$ and edge weight function $\om_{\se}:V_\se\times V_\se\to \R_+$ both constructed by induction as follows:
\begin{enumerate}
\item[{\tiny $\bullet$}] ${\sf v}^0_{i_0,{\bf 0}}$ is the unique vertex of $V_{\se}^0$

\item[{\tiny $\bullet$}] if ${\sf v}^k_{i,\be}\in V_{\se}^\ell$ with $\ell\geq 0$ and $\om(v_i,v_j)[\de]\in \R^\ast_+$ for some $\de\in \Z^N$ then ${\sf v}_{j,\be+\de}^{\ell+1}\in V_{\se}^{\ell+1}$
and $\om_{e}({\sf v}^k_{i,\be},{\sf v}_{j,\be+\de}^{k+1}) = \om(v_i,v_j)[\de]$.
\end{enumerate}
\end{definition}
\begin{remark}
\begin{enumerate}
\item By construction of $\Ga_\se$, for all $i\in \{1,\ldots,n\}$, $\be\in \Z^N$ and $\ell\in \N$, we have $\sv_{i,\be}^\ell\in V_\se$ if and only if there is a path $\pi$ in $\Ga$ of length $\ell$ from $v_{i_0}$ to $v_i$ such that $\wt(\pi)[\be]\neq 0$. Moreover, given $\sv_{i,\be}^\ell\in V_\se^\ell$, we have $\sv_{j,\de}^{\ell+1}\in V_\se^{\ell+1}$ if and only if we have $\om(v_i,v_j)[\de-\be]\neq 0$ (in particular there is an edge from $v_i$ to $v_j$).
\item There is a strong connection between the edge weight functions $\om_\se$ and $\om$.  More precisely, for any vertex $\sv_{i,\be}^\ell\in V_\se^\ell$ we have
\begin{align*}
\sum_{\sv_{j,\de}^{\ell+1}\in V_\se^{\ell+1}} \om_\se(\sv_{i,\be}^\ell,\sv_{j,\de}^{\ell+1})z^{\de-\be}
&=\sum_{\sv_{j,\de}^{\ell+1}\in V_\se^{\ell+1}} \om(v_i,v_j)[\de-\be]z^{\de-\be}\\
&=\sum_{j\in \{1,\ldots,n\}}\sum_{\de\in \Z^N}\om(v_i,v_j)[\de-\be]z^{\de-\be}\tag{using (1)}\\
&= \sum_{j\in \{1,\ldots,n\}} \om(v_i,v_j).
\end{align*}
\end{enumerate}
\end{remark}
\begin{example}
Let $\Gamma$ be the graph given by 
\begin{tikzpicture}[scale =1,baseline=-.1cm]
\tikzstyle{vertex}=[inner sep=2pt,minimum size=10pt,circle,draw]
\node[vertex] (a1) at (0,0){$v_1$};
\node[vertex] (a2) at (1.5,0){$v_2$};

\draw[line width=.5pt,->,bend left = 20] (a1) to node[above]{$1$} (a2);
\draw[line width=.5pt,->,,bend left = 20] (a2) to node[below]{$z$}  (a1);

\end{tikzpicture}.
The expansion $\Ga_{\se}$ of $\Ga$ at $v_1$ is
$$\begin{tikzpicture}[scale =1]
\tikzstyle{vertex}=[inner sep=2pt,minimum size=10pt]
\node[vertex] (a1) at (0,0){$\sv^0_{1,0}$};
\node[vertex] (a2) at (1.5,0){$\sv^1_{2,0}$};
\node[vertex] (a3) at (3,0){$\sv^2_{1,1}$};
\node[vertex] (a4) at (4.5,0){$\sv^3_{2,1}$};
\node[vertex] (a5) at (6,0){$\sv^4_{1,2}$};
\node[vertex] (a6) at (7.5,0){};

\draw[line width=.5pt,->,bend left = 0] (a1) to node[above]{$1$} (a2);
\draw[line width=.5pt,->,bend left = 0] (a2) to node[above]{$1$} (a3);
\draw[line width=.5pt,->,bend left = 0] (a3) to node[above]{$1$} (a4);
\draw[line width=.5pt,->,bend left = 0] (a4) to node[above]{$1$} (a5);
\draw[->,dashed] (a5) to node[above]{$1$} (a6);

\end{tikzpicture}$$
\end{example}

Assume that $\Gamma$ be a positively multiplicative with respect to the matrix realisation $(\cA,\mB)$ where $\mB=\{\mb_{1}=1,\mb_{2},\ldots
,\mb_{n}\}$. By definition $\mathcal{A}$ is an algebra over $\K(\sZ^{\pm1})$ whose structure constants $c_{i,j}^{k}$ with respect to $\mB$ belong to  $\R_+[\sZ^{\pm1}]$. In other words, we have
\[
c_{i,j}^{k}=\sum_{\beta\in\mathbb{Z}^{N}}c_{i,j}^{k}[\beta]z^{\beta}\qu{where $\beta=(\beta_1,\ldots,\beta_N)\in \Z^N$ and $c_{i,j}^{k}[\beta]\in \R_+$.}
\]
The algebra $\mathcal{A}$ can also be viewed as an infinite dimensional $\K$-algebra $
\cA_\K$ with basis $$\{z^{\beta}\mb_{i}\mid i=1,\ldots,n \text{ and }\beta\in\mathbb{Z}^{N}\}.$$
In order to extend the notion of positively
multiplicative graphs to the expansion $\Gamma_{\se}$ of $\Gamma$, we will in
fact need the larger algebra $\cA'_\se=\mathcal{A}_{\K}\underset{\mathrm{\K}}{\otimes}\K[q]$ where $q$ is a new
indeterminate distinct from $z_{1},\ldots,z_{N}$. The powers of~$q$ will
record the lengths of the paths starting from the vertex $v_{i_0}$ of $\Gamma$ choosen to construct the expansion. The set
$$\mB_{\se}'=\{q^{\ell
}z^{\beta}\mb_{i}\mid i=1,\ldots,n,\beta\in\mathbb{Z}^{N},\ell\in\mathbb{N}%
\}$$
is a $\K$-basis of $\cA'_e$. 
 Finally we define $\mathcal{A}_{\se}$ to be the subspace of $\mathcal{A}'
_{\se}$ with basis $\mB_{\se}=\{q^{\ell}z^{\beta}\mb_{i}\mid
\sv_{i,\be}^\ell\in V_{\se}\}$. 

\smallskip

Given $\sv_{i,\be}^\ell$ where $i\in \{1,\ldots,n\}$, $\be\in \Z^N$ and $\ell\in\N$, we set $\mb_{i,\be}^\ell = q^\ell z^\beta \mb_i\in\mB'_\se$. Note that $1=\mb_{1,{\bf 0}}^0$.

\begin{proposition}
\label{Prop_Gammae_PM} Let $\Ga_\se$ be the expansion of $\Ga$ at $v_1$. 
\begin{enumerate}
\item For any vertex $\sv_{j,\beta}^\ell\in V_{\se}$, we have in $\cA_{\se}$
\[
qA_\Ga\times \mb_{j,\be}^\ell=
\sum_{{\sf v}^{\ell+1}_{i,\de}\in V_{\ell+1}} \om_\se({\sf v}^\ell_{j,\be},{\sf v}^{\ell+1}_{i,\de})\mb_{i,\de}^{\ell+1}
\]
In particular, $qA_\Ga\in \mB_{\se}'$.
\item $\mB_{\se}$ is the subset of $\mB_{\se}'$
containing the elements $q^{\ell}z^{\beta}b_{i}$ with $\ell\geq0$ which appear with a
non-zero coefficient in the expansions of the powers $(qA_\Ga)^{\ell},\ell\geq0$
on the basis $\mB_{\se}'$.
\item The element $\mb_{1,{\bf 0}}^0=1$ belongs to $\mB_{\se}$ and the
product of two elements in the basis $\mB_{\se}$ expands on
$\mB_{\se}$ with nonnegative real coefficients. In particular
$\mathcal{A}_{\se}$ is a subalgebra of $\mathcal{A}_{\se}'$ with PM basis
$\mB_{\se}$.
\end{enumerate}
\end{proposition}

\begin{proof}
We prove (1). Let $\sv_{j,\be}^\ell\in V_\se$.  We have%
\begin{align*}
qA_\Ga\times \mb_{j,\be}^\ell&=q^{\ell+1}z^{\beta}A_\Ga\mb_{j}\\
&=q^{\ell+1}z^{\beta}\sum_{i=1}^{n}\om(v_j,v_i)\mb_{i}\\
&=q^{\ell+1}z^{\beta}\sum_{i=1}^{n}\sum_{\de\in \Z^N} \om(v_j,v_i)[\de]z^{\de}\mb_{i}\\
&=\sum_{i=1}^{n}\sum_{\de\in \Z^N} \om(v_j,v_i)[\de]q^{\ell+1}z^{\beta+\de}\mb_{i}\\
&=\sum_{i=1}^{n}\sum_{\de'\in \Z^N} \om(v_j,v_i)[\de'-\beta]\mb_{i,\de'}^{\ell+1}.
\end{align*}
But $\om(v_j,v_i)[\de'-\beta]\neq 0$ exactly when there is an edge from ${\sf v}^{\ell}_{j,\be}$ to ${\sf v}^{\ell+1}_{i,\de'}$ of weight $\om(v_j,v_i)[\de'-\beta]$. Hence the result. 

\smallskip

We prove (2). Since $\mb_{1,1}^0=1\in V_{\se}$, we get by induction using (1):
$$(qA_\Ga)^k = \underset{\text{$\pi$ starts at ${\sf v}^0_{1,1}$ and ends at ${\sf v}^k_{i,\be}$}}{\sum_{\pi, \ell(\pi) = k}}\wt(\pi) \mb_{i,\be}^k$$
We have  $\mb_{i,\be}^k\in \mB_{\se}$ (or equivalently $\sv_{i,\be}^k\in V_{\se}$) if and only if there exists a path from ${\sf v}^0_{1,1}$ to ${\sf v}^k_{i,\be}$ in~$\Ga_\se$. The result follows. 

\smallskip

We prove (3).  Given $\mb_{i,\be}^\ell$ and $\mb_{j,\ga}^s$ in the basis
$\mB_{\se}'$, we have
\[
\mb_{i,\be}^k\times \mb_{j,\ga}^s=q^{\ell}z^{\beta}\mb_{i}\times q^{s}z^{\gamma}\mb_{j}
=\sum_{k=1}^{n}c_{i,j}^{k}q^{\ell+s}z^{\beta+\gamma}\mb_{k}
=\sum_{k=1}^{n}\sum_{\delta\in\Z^N}c_{i,j}^{k}[\delta]\mb^{\ell+s}_{k,\delta+\beta+\gamma}
\]
which shows that $\mb_{i,\be}^\ell\times \mb_{j,\ga}^s$ expands positively on
$\mB_{\se}'$. Now if $\mb_{i,\be}^\ell$ and $\mb_{j,\ga}^s$ lie in $\mB_\se$, then $\mb_{i,\be}^\ell$ appears in $(qA_\Ga)^\ell$ and $\mb_{j,\ga}^s$ appears in $(qA_\Ga)^s$. Since $A_\Ga\in \func{Mat}_{n}(\R_+[\sZ^{\pm 1}])$, it follows that all the elements that appear in the product $\mb_{i,\be}^\ell\times \mb_{j,\ga}^s$ actually appears in $(qA_\Ga)^{\ell+s}$ and hence belong to $\mB_\se$. \end{proof}

\begin{remark}
\label{Rem_Simply}In fact, the interesting finite graphs $\Gamma$ have often
additional properties which simplify the definition and the study of their
expansion . This is the case when

\begin{enumerate}
\item The weights of $\Gamma$ are only nonnegative reals.
\item The length of a path is completely determined by its weight and by its end point in $\Gamma$.  In this
case, the indeterminate $q$ in the previous construction is redundant and can
be omitted. This is the case for the graph considered in the example below.
\end{enumerate}
\end{remark}

\begin{example}
The graph $\Gamma$%

\begin{center}
\begin{tikzpicture}[scale =1.3]
\tikzstyle{vertex}=[inner sep=2pt,minimum size=10pt,ellipse,draw]
\node[vertex] (a0) at (0,0) {$v_1$};
\node[vertex] (a1) at (0,-1) {$v_2$};

\draw[->] (a0) -- (a1);
\draw[->,bend left=55] (a1) to node[left]{$z_1$} (a0);
\draw[->,bend right=55] (a1) to node[right]{$z_2$} (a0);

\end{tikzpicture}
\end{center}
with adjacency matrix $A_\Ga=\left(
\begin{array}
[c]{cc}%
0 & z_{1}+z_{2}\\
1 & 0
\end{array}
\right)  $ is PM 
 with respect to the matrix realisation $(\cA,\mB)$ where
$\mathcal{A}=\mathbb{R}I_{2}\oplus\mathbb{R}A$ and $\mB=\{I_{2},A\}$. We draw the graph $\Gamma_{\se}$ of the expansion of $\Gamma$ at $v_1$ below on the left handside.  The indeterminate $q$ can be omitted, since when
$\sv_{i,\be}^\ell\in\Gamma_{\se}$ with $\be=(\be_1,\be_2)$ we must have $\ell=2\left(  \beta_{1}+\beta
_{2}\right)  +1$ if $v=v_2$ and $\ell=2(\beta_{1}+\beta_{2})$ if
$v=v_1$. We have 
$$\mathcal{A}_{\se}=\mathbb{R[}z_{1},z_{2}]I_{2}\oplus\mathbb{R[}z_{1},z_{2}]A.$$ 
Note that the graph $\Gamma_{\se}$ can also be labeled by $2$-bounded partitions (that is partition with parts less than or equal to 2) as illustrated on the right handside.

\begin{figure}[H]
\begin{minipage}{7.5cm}
\begin{center}
\begin{tikzpicture}[scale =1]
\def\rx{0.5}
\def\ry{0.866}
\tikzstyle{vertex}=[inner sep=2pt,minimum size=10pt]

\node[vertex] (a0) at (0,0) {$\sv^0_{1,(0,0)}$};
\node[vertex] (a1) at (0,-1) {$\sv^1_{2,(0,0)}$};
\node[vertex] (a21) at (-1,-2) {$\sv^2_{1,(1,0)}$};
\node[vertex] (a22) at (1,-2) {$\sv^2_{1,(0,1)}$};

\node[vertex] (a31) at (-1,-3) {$\sv^3_{2,(1,0)}$};
\node[vertex] (a32) at (1,-3) {$\sv^3_{2,(0,2)}$};

\node[vertex] (a41) at (-2,-4) {$\sv^4_{1,(2,0)}$};
\node[vertex] (a42) at (0,-4) {$\sv^4_{1,(1,1)}$};
\node[vertex] (a43) at (2,-4) {$\sv^4_{1,(0,2)}$};

\node[vertex] (a51) at (-2,-5) {$\sv^5_{2,(2,0)}$};
\node[vertex] (a52) at (0,-5) {$\sv^5_{2,(1,1)}$};
\node[vertex] (a53) at (2,-5) {$\sv^5_{2,(0,2)}$};

\node[vertex] (a61) at (-3,-6) {$\sv^6_{1,(3,0)}$};
\node[vertex] (a62) at (-1,-6) {$\sv^6_{1,(2,1)}$};
\node[vertex] (a63) at (1,-6) {$\sv^6_{1,(1,2)}$};
\node[vertex] (a64) at (3,-6) {$\sv^6_{1,(0,3)}$};

\draw[line width=.5pt,->] (a0) to   (a1);
\draw[line width=.5pt,->] (a1) to  (a21);
\draw[line width=.5pt,->] (a1) to  (a22);

\draw[line width=.5pt,->] (a21) to  (a31);
\draw[line width=.5pt,->] (a22) to  (a32);

\draw[line width=.5pt,->] (a31) to  (a42);
\draw[line width=.5pt,->] (a32) to  (a42);
\draw[line width=.5pt,->] (a31) to  (a41);
\draw[line width=.5pt,->] (a32) to  (a43);

\draw[line width=.5pt,->] (a41) to  (a51);
\draw[line width=.5pt,->] (a42) to  (a52);
\draw[line width=.5pt,->] (a43) to  (a53);

\draw[line width=.5pt,->] (a41) to  (a51);
\draw[line width=.5pt,->] (a42) to  (a52);
\draw[line width=.5pt,->] (a43) to  (a53);

\draw[line width=.5pt,->] (a51) to  (a61);
\draw[line width=.5pt,->] (a51) to  (a62);
\draw[line width=.5pt,->] (a52) to  (a62);
\draw[line width=.5pt,->] (a52) to  (a63);
\draw[line width=.5pt,->] (a53) to  (a63);
\draw[line width=.5pt,->] (a53) to  (a64);
\end{tikzpicture}
\end{center}
\end{minipage}
\begin{minipage}{7.5cm}
\begin{center}
\begin{tikzpicture}[scale =1]
\def\rx{0.5}
\def\ry{0.866}
\tikzstyle{vertex}=[inner sep=2pt,minimum size=10pt]

\node[vertex] (a0) at (0,0) {\scalebox{1}{$\varnothing$}};
\node[vertex] (a1) at (0,-1) {$\scalebox{.2}{\ydiagram{1}}$};
\node[vertex] (a21) at (-1,-2)  {$\scalebox{.2}{\ydiagram{1,1}}$};
\node[vertex] (a22) at (1,-2)  {$\scalebox{.2}{\ydiagram{2}}$};

\node[vertex] (a31) at (-1,-3) {$\scalebox{.2}{\ydiagram{1,1,1}}$};
\node[vertex] (a32) at (1,-3)  {$\scalebox{.2}{\ydiagram{2,1}}$};

\node[vertex] (a41) at (-2,-4){$\scalebox{.2}{\ydiagram{1,1,1,1}}$};
\node[vertex] (a42) at (0,-4) {$\scalebox{.2}{\ydiagram{2,1,1}}$};
\node[vertex] (a43) at (2,-4) {$\scalebox{.2}{\ydiagram{2,2}}$};

\node[vertex] (a51) at (-2,-5) {$\scalebox{.2}{\ydiagram{1,1,1,1,1}}$};
\node[vertex] (a52) at (0,-5) {$\scalebox{.2}{\ydiagram{2,1,1,1}}$};
\node[vertex] (a53) at (2,-5) {$\scalebox{.2}{\ydiagram{2,2,1}}$};

\node[vertex] (a61) at (-3,-6)  {$\scalebox{.2}{\ydiagram{1,1,1,1,1,1}}$};
\node[vertex] (a62) at (-1,-6)  {$\scalebox{.2}{\ydiagram{2,1,1,1,1}}$};
\node[vertex] (a63) at (1,-6)  {$\scalebox{.2}{\ydiagram{2,2,1,1,1}}$};
\node[vertex] (a64) at (3,-6)  {$\scalebox{.2}{\ydiagram{2,2,2}}$};

\draw[line width=.5pt,->] (a0) to   (a1);
\draw[line width=.5pt,->] (a1) to  (a21);
\draw[line width=.5pt,->] (a1) to  (a22);

\draw[line width=.5pt,->] (a21) to  (a31);
\draw[line width=.5pt,->] (a22) to  (a32);

\draw[line width=.5pt,->] (a31) to  (a42);
\draw[line width=.5pt,->] (a32) to  (a42);
\draw[line width=.5pt,->] (a31) to  (a41);
\draw[line width=.5pt,->] (a32) to  (a43);

\draw[line width=.5pt,->] (a41) to  (a51);
\draw[line width=.5pt,->] (a42) to  (a52);
\draw[line width=.5pt,->] (a43) to  (a53);

\draw[line width=.5pt,->] (a41) to  (a51);
\draw[line width=.5pt,->] (a42) to  (a52);
\draw[line width=.5pt,->] (a43) to  (a53);

\draw[line width=.5pt,->] (a51) to  (a61);
\draw[line width=.5pt,->] (a51) to  (a62);
\draw[line width=.5pt,->] (a52) to  (a62);
\draw[line width=.5pt,->] (a52) to  (a63);
\draw[line width=.5pt,->] (a53) to  (a63);
\draw[line width=.5pt,->] (a53) to  (a64);

\end{tikzpicture}
\end{center}
\end{minipage}
\caption{The graph $\Gamma_{\se}$ and its labeling by 2-bounded partitions}
\end{figure}

\end{example}

\subsection{Infinite PM graphs and harmonic functions}
\newcommand{\eGa}{\mathsf{\Ga}}
\newcommand{\eV}{\mathsf{V}}
\newcommand{\eE}{\mathsf{E}}
\label{SubsecHarmonic} In this section, $\eGa$ denotes 
an infinite rooted graded (oriented) graph with set of vertices $\eV$ where $\eV$ is countable, edge weight function~$\und{\om}$ that takes values in $\R_+$ and such that $v_1$ is the unique vertex at level $0$.

\begin{definition}
\label{infinite_mult}
The infinite graph $\eGa$ is  \emph{positively multiplicative} when there exists an 
algebra $\sf A$ over $\K$, PM with respect to a basis $\mB=\{\mb_{v}\mid v\in \eV\}$  where $\mb_{v_{1}}=1$ and a distinguished element $\gamma$ in $\sf A$ such that
\[
\gamma \mb_{v}=\sum_{v'\in \eV}\und{\om}(v,v')\mb_{v'}\text{ for any }v\in\eV.
\]
In particular, we have \footnote{Note that $\gamma$ is given by the fact that we must have 
$\gamma = \gamma \mb_{v_1} = \sum_{v'\in V}\und{\om}(v_1,v')\mb_{v'}$.}
\[
\gamma=\sum_{v\in \eV}\und{\om}(v_1,v)\mb_{v}.
\]
\end{definition}

\begin{definition}
A nonnegative (respectively positive) harmonic function on $\eGa$ is a
map $f:\eV\mathbb{\rightarrow}\mathbb{R}_{+}$ such that
$f(v_{1})=1$ and for any vertex $v\in\eV$%
\[
f(v)\geq0\text{ (resp. $f(v)>0$) \quad and\quad}f(v)=\sum_{v'\in \eV}\om_\se(v,v')f(v').
\]
\end{definition}
Observe that we have then%
\[
1=f(v_{1})=\sum_{v'\in \eV}\und{\om}(v_1,v')f(v').
\]

We denote by $\mathcal{H}(\eGa)$ the set of nonnegative
harmonic functions on $\eGa$. This is a convex cone and we write
$\mathcal{H}_{\partial}(\eGa)$ for its subset of extremal
points.

\begin{remark}
Positive harmonic functions on $\eGa$ are strongly connected to
Markov chains.\ Indeed, if $f$ is a positive harmonic function, we can define a Markov
$\mathcal{H}$ chain on $\eGa$ with transition matrix $\Pi$ defined by:
\[
\Pi(v,v')=\und{\om}(v',v)\frac{f(v')}{f(v)}\qu{for all $v,v'\in \eV$.}
\]
It then becomes possible to study $\mathcal{H}$ (for example to get its drift,
a law of large numbers etc.) when~$f$ is sufficiently simple, in particular
when $f$ is extremal (see \cite{Ker}, \cite{LLP2} and \cite{LT2} for examples).
\end{remark}

Recall the following theorem essentially due to Kerov and Vershik
(see \cite{Ker} and also \cite{LT}).

\begin{theorem}
\label{Th_KV}Assume that $\K(\sZ^{\pm 1})=\mathbb{R}$ and that $\eGa$ is
positively multiplicative with associated algebra ${\sf A}$ and basis
$\mB=\{b_{v}\mid v\in\eV\}$ with $1=\mb_{v_1}\in\mB$. Then the
map $f:\eV\rightarrow\mathbb{R}_{+}$ belongs to
$\mathcal{H}_{\partial}(\eGa)$ if and only 
the linear form
$\varphi:\mathbb{A}\rightarrow\mathbb{R}$ defined by $\varphi(\mb_{v})=f(v)\text{ for any }v\in\eV$
is a morphism of $\mathbb{R}$-algebras satisfying~$\varphi(\ga)=1$.
\end{theorem}

\begin{example}
Let $\eGa= (\eV,\eE,\und{\om})$ where $\eV=\{(a,b)\in \N^2\mid a=0 \text{ or } b\leq 1\}$, 
$$\eE = \{\big((a,b),(a',b')\big)\in \eV\times\eV\mid (a',b') = (a+1,b)\text{ or }(a',b') = (a,b+1)\}$$
and $\und{\om}$ is constant equal to $1$ on $\eE$. This graph is positively multiplicative with respect to the algebra 
$${\sf A} = \R[x_1,x_2]\slash \langle x_1x_2^2\rangle \qu{and the basis} \mB = \{\overline{x_1}^a\overline{x_2}^b\mid (a,b)\in \eV\}$$ where $\overline{x_1}$ and $\overline{x_2}$ denotes the images of $x_1$ and $x_2$ in ${\sf A}$. Let $f$ be an extremal harmonic function on $\eGa$. Then according the theorem above, we have $f(\ga)=f(1,0)+f(0,1)=1$ and  the map $\varphi$ from ${\sf A}$ to $\R_+$ defined by $\overline{x_1}^a\overline{x_2}^b\to f(a,b)$ is a morphism of algebras. In particular  $\overline{x_1x_2^2}$ is sent to $0$. But 
$$\varphi(\overline{x_1x_2^2})=\varphi(\overline{x_1})\varphi(\overline{x_2})^2=f(1,0)f(0,1)=0$$ 
hence $f(1,0)=0$ ou $f(0,1)=0$.  Finally we have
$$\mathcal{H}_{\partial}(\eGa) = \{f_1,f_2\}\text{ where $f_1(1,0) =1$, $f_1(0,1)=0$ and $f_2(1,0)=0$ and $f_2(0,1)=1$.}$$

\smallskip

Let us verify the theorem above "by hands". Let $f$ be an harmonic function on $\eGa$. We have $f(1,0)+f(0,1)=1$ and $f$ is positive so that $f(1,0)=p\in [0,1]$ and $f(0,1)=1-p\in[0,1]$. Next we set $f(2,0) =q\in[2p-1,q]$. Below, we draw the graph $\eGa$ and we put the expected values of $f$ in red close to each vertex. 
\begin{center}
\begin{tikzpicture}[scale =1]
\def\rx{0.65}
\def\ry{1}
\tikzstyle{vertex}=[inner sep=2pt,minimum size=10pt]

\node[vertex] (a0) at (0,0) {$(0,0)$};
\node at (-.6,0) {{\red $1$}};

\node[vertex] (a1) at (-\rx,-\ry) {$(1,0)$};
\node at (-\rx-.6,-\ry) {{\red $p$}};

\node[vertex] (a2) at (\rx,-\ry) {$(0,1)$};
\node at (\rx+1,-\ry) {{\red $1-p$}};

\node[vertex] (a3) at (-2*\rx,-2*\ry) {$(2,0)$};
\node at (-2*\rx-.6,-2*\ry) {{\red $q$}};

\node[vertex] (a4) at (0,-2*\ry) {$(1,1)$};
\node at (.4,-2*\ry-.4) {{\red $p-q$}};

\node[vertex] (a5) at (2*\rx,-2*\ry) {$(0,2)$};
\node at (2*\rx+1.35,-2*\ry) {{\red $1-2p+q$}};

\node[vertex] (a6) at (-3*\rx,-3*\ry) {$(3,0)$};
\node at (-3*\rx-1.1,-3*\ry) {{\red $2q-p$}};

\node[vertex] (a7) at (-\rx,-3*\ry) {$(2,1)$};
\node at (-\rx+.4,-3*\ry-.4) {{\red $p-q$}};

\node[vertex] (a8) at (3*\rx,-3*\ry) {$(0,3)$};
\node at (3*\rx+1.35,-3*\ry) {{\red $1-2p+q$}};

\node[vertex] (a9) at (-4*\rx,-4*\ry) {$(4,0)$};
\node at (-4*\rx-1.1,-4*\ry) {{\red $3q-p$}};

\node[vertex] (a10) at (-2*\rx,-4*\ry) {$(3,1)$};
\node at (-2*\rx+.4,-4*\ry-.4) {{\red $p-q$}};

\node[vertex] (a11) at (4*\rx,-4*\ry) {$(0,4)$};
\node at (4*\rx+1.35,-4*\ry) {{\red $1-2p+q$}};

\draw[line width=.5pt,->] (a0) to   (a1);
\draw[line width=.5pt,->] (a0) to   (a2);

\draw[line width=.5pt,->] (a1) to   (a3);
\draw[line width=.5pt,->] (a1) to   (a4);
\draw[line width=.5pt,->] (a2) to   (a5);
\draw[line width=.5pt,->] (a2) to   (a4);

\draw[line width=.5pt,->] (a3) to   (a6);
\draw[line width=.5pt,->] (a3) to   (a7);
\draw[line width=.5pt,->] (a4) to   (a7);

\draw[line width=.5pt,->] (a6) to   (a9);
\draw[line width=.5pt,->] (a6) to   (a10);
\draw[line width=.5pt,->] (a7) to   (a10);

\draw[line width=.5pt,->] (a5) to   (a8);
\draw[line width=.5pt,->] (a8) to   (a11);

\end{tikzpicture}
\end{center}

It is a straightforward exercise to show that $f$ satisfies 
$f(k,0)= (k-1)q-(k-2)p\geq 0$ for all $k\geq 2$. Using the fact that $f$ is positive, we get  $p\in [0,1]$ and $q\in \left[ \dfrac{k-2}{k-1}p,p\right]$
for all $k\geq 2$. It follows that  $q=p$ and $f(1,1) = 0$.  Finally, we get $f=pf_1+(1-p)f_2$ as expected.
\end{example}

\begin{example}
Let $\mathcal{Y}_{n}$ be the Young lattice of partitions with at most $n$
parts. Recall that its vertices are the Young diagrams $\lambda$ associated to
the partitions with at most $n$ parts and we have an edge $\lambda
\rightarrow\mu$ if the Young diagram of~$\mu$ is obtained by adding a box to the Young diagram of~$\lambda$. Let $\mathbb{A}=\mathrm{Sym}[x_{1},\ldots,x_{n}]$ be the
$\mathbb{R}$-algebra of symmetric functions in the indeterminates~$x_{1},\ldots,x_{n}$ and $\mB=\{s_{\lambda}\mid\lambda
\in\mathcal{Y}_{n}\}$ where $s_{\lambda}$ is the Schur function associated to
$\lambda$.\ Then~$\mathcal{Y}_n$ is positively multiplicative with respect to $\mathbb{A}$
and the basis~$\mB$. Its extremal positive harmonic functions
are the morphisms $\varphi:\mathbb{A}\rightarrow\mathbb{R}$ such that
$\varphi(s_{\lambda})\in\mathbb{R}_{+}^\ast$.\ One can show in this case that
these morphisms are parametrised by the vectors $p=(p_{1},\ldots,p_{n}%
)\in\mathbb{R}_{+}^{m}$ with $p_{1}+\cdots+p_{n}=1$, the morphism
$\varphi_{p}$ corresponding to $p$ being the specialisation $x_{i}%
=p_{i},i=1,\ldots,n$. The structure constants $c_{\lambda,\mu}^{\nu}$ are
the Littlewood-Richardson coefficients and it is a classical result that~$c_{\lambda
,\mu}^{\nu}>0$ only if $\nu$ can be reached from $\lambda$ by a path in
$\mathcal{Y}_{n}$ (i.e. $\lambda\subset\nu$). 
\end{example}

\subsection{Positive harmonic functions on expanded graphs}
\newcommand{\Gr}{\mathsf{Graph}}
\label{SubsecPMonExpansion} 

In this section, $\Gamma\in \mathsf{Graph}_n(\R_+[\sZ^{\pm1}])$ with set of vertices $\{v_1,\ldots,v_n\}$ is a strongly connected and 
positively multiplicative graph with respect to the matrix realisation $(\cA,\mB)$  where $\mB=(\mb_1,\ldots,\mb_n)$ and $\mb_1=1$. 
Let  $\Gamma_{\se}$ be the expansion of $\Ga$ at $v_1$ with set of vertices~$V_\se$.
We denote by $\om_\se$ the edge weight function on  $\Ga_\se$.
\begin{proposition}
\label{Gae_mult}
The infinite graph $\Gamma_{\se}$ is positively multiplicative with respect to the algebra $\mathcal{A}_{\se}$ and the basis~$\mB_{\se}$. 
\end{proposition}
\begin{proof}
First of all, the graph $\Gamma_{\se}$ is rooted and graded with $\sv^0_{1,0}$ the unique vertex at level $0$. We have $\mb^0_{1,0}= q^{0}z^0\mb_{1}=1$ and by Proposition \ref{Prop_Gammae_PM}.(1), we see that the element $\ga$ in Definition \ref{infinite_mult} is:
\[
qA_\Ga=qA_\Ga\times \mb^0_{1,0}= \sum_{{\sf v}^{1}_{i,\de}\in V_\se^{1}} \om_\se(\sv^0_{1,0},{\sf v}^{1}_{i,\de})\mb_{i,\de}^{1}.
\]
The result then follows from Proposition \ref{Prop_Gammae_PM}.
\end{proof}
Given $\bt= (t_1,\ldots,t_N)\in \R_+^\ast$ and $(\be_1,\ldots,\be_N)\in \Z^N$ we write $\bt^{\be} = t_1^{\be_1}\ldots t_{N}^{\be_N}$. With this notation, the specialisation that sends $z_k$ to $t_k$ for all $k\in \{1,\ldots,N\}$  sends $z^\be$ to $\bt^\be$. We write $A_{\bt}$ for the matrix obtained from $A_\Ga$ by applying this specialisation to all the coefficients of $A_\Ga$. 
\begin{theorem}
Assume that $\K[\sZ^{\pm1},q]\subset\mathcal{A}_{\se}$ and $\mathcal{A}=\K(\sZ^{\pm1})[A_\Ga]$ (i.e. $\Gamma$ has maximal dimension). Then, the set
$\mathcal{H}_{\partial}^{+}(\Gamma_{\se})$ of extremal positive harmonic
functions is parametrised by a subset of~${\R_{+}^\ast}^{N}$. More
precisely, to any $\bt =(t_{1},\ldots,t_{N})$ in~${\R_{+}^\ast}^{N}$
corresponds an extremal harmonic function $\varphi$ on $\Gamma_{e}$ such that
\[
\varphi(\sv_{i,\be}^\ell)=t^{\beta}\lambda^{-\ell}\pi_{i}\text{ for
}i=1,\ldots,n
\]
where $\pi=(\pi_{1},\ldots,\pi_n)$ is the left Perron-Frobenius vector of the matrix $A_\bt$ with eigenvalue
\[
\lambda=\sum_{i=1}^n \varphi(a_{i,1})\pi_{i}.
\]
Moreover, all the elements in $\mathcal{H}_{\partial}^{+}(\Gamma_{e})$ are
obtained in this way.
\end{theorem}

\begin{proof}
By Theorem \ref{Th_KV}, the elements of $\mathcal{H}_{\partial}^{+}%
(\Gamma_{\se})$ are determined by the morphisms of $\R$-algebras  $\varphi:\mathcal{A}%
_{e}\rightarrow\mathbb{R}$ which are positive on the basis $\mB_{e}$ and such
that $\varphi(\gamma)=1$ where $\gamma=qA_\Ga$.

\smallskip

Consider such a morphism $\varphi$. First of all, since $\K[\sZ^{\pm1},q]\subset\mathcal{A}_{\se}$, we have $\mb_i\in \cA_e$ for all $i=1,\ldots,N$. Since $(\cA,\mB)$ is a matrix realisation of $\Ga$ and $\mb_1=1$, we have $A_\Ga = \sum_{i=1}^na_{i,1}\mb_i$ so that $A_\Ga\in \cA_e$. 
We set $\varphi(q)=\lambda^{-1}\in\mathbb{R}_{>0}$, $\varphi(z_{k})=t_{k}\in\mathbb{R}_{>0}$ and $\varphi(\mb_i)=\pi_i\in \R$ for each $k=1,\ldots,N$. We have $\varphi(A_\Ga) = \varphi(q^{-1}) = \la$ since $\varphi(\ga)=1$. Note that $\la =\sum_{i=1}^n \varphi(a_{i,1})\pi_{i} $.

\smallskip

We show that $\pi_i\in \R_+^\ast$ for all $1\leq i\leq n$. Let $i\in \{1,\ldots,n\}$. Since $\Ga$ is strongly connected, there exists a path from $v_1$ to $v_i$ in $\Ga$. Thus,  there is a vertex of the form $\sv_{i,\be}^\ell$ in $\Ga_\se$ for some $\be\in \Z^N$ and $\ell\in \N$. We have $\varphi(\mb_{i,\be}^\ell) = \varphi(q^\ell z^\be\mb_{i})=\la^{-\ell}\bt^\be\pi_i$ so that $\pi_i = \dfrac{\varphi(\mb_{i,\be}^\ell)}{\la^{-\ell}\bt^\be}\in \R_+^\ast$ since $\la,\bt^\be,\varphi(\mb_{i,\be}^\ell)$ all lie in $\R_+^\ast$ ($\varphi$ is positive on $\mB_\se$). We note here  that if $\sv_{i,\be'}^{\ell'}\in V_\se$ then 
$$\varphi(q^{\ell+\ell'}z^{\be+\be'}\mb_i) =\varphi(\mb_{i,\be}^\ell)\lambda^{-\ell'}t^{\be'}=\varphi(\mb_{i,\be'}^{\ell'})\lambda^{-\ell}t^{\be}\qu{and }\frac{\varphi(\mb_{i,\be}^\ell)}{\lambda^{-\ell}t^\be}=\frac{\varphi(\mb_{i,\be'}^{\ell'})}{\lambda
^{-\ell'}t^{\be'}}=\pi_i.$$
The coefficient of $A_\bt$ (the matrix $A$ in which each $z_i$ is specialised at $t_i$) in position $(i,j)$ is $\varphi(a_{i,j})$ (recall that $a_{i,j}\in \K[\sZ^{\pm 1}]$). 
Let $1\leq j\leq n$. 
We have
\begin{align*}
qA_\Ga\mb_{j}=q\sum_{i=1}^{n}a_{i,j}\mb_{i}.
\end{align*}
It follows that $\varphi(qA_\Ga) \pi_j= \varphi(q)\sum_{i=1}^{n}\varphi(a_{i,j})\pi_i$ and 
$\sum_{i=1}^{n}\varphi(a_{i,j})\pi_i = \la\pi_j$. Hence showing that $\la$ is an left eigenvalue of $A_\bt$ associated to the left eigenvector $(\pi_1,\ldots,\pi_n)$. The matrix $A_{\boldsymbol{t}}$ is still
irreducible (because the $t_{i}$'s are strictly positive), therefore $\pi$ is equal to the
unique left Perron-Frobenius vector of $A_{\boldsymbol{t}}$ and $\varphi
(A_\Ga)=\lambda$ is the associated Perron-Frobenius eigenvalue. The corresponding
positive extremal function~$f$ is defined on $\Gamma_{e}$ by $f(\sv_{i,\be}^\ell)=\varphi(\mb_{i,\be}^\ell)=\lambda^{-\ell}t^{\beta
}\pi_{i}$.

\medskip

Conversely, for any $\boldsymbol{t}=(t_{1},\ldots,t_{N})\in\mathbb{R}_{>0}%
^{N}$, let $(\pi_{1}=1,\pi_{2},\ldots,\pi
_{n})$ be the normalised Perron-Frobenius vector of $A_\bt$. We have a morphism
$\varphi:\K[\sZ^{\pm1}]\rightarrow\mathbb{R}$ defined by $\varphi
(z_{k})=t_{k}$ for any $k=1,\ldots,N$. By Lemma \ref{Lem_IntegralDomain}, we know that $\cA=\K[\sZ^{\pm1}][A_\Ga]$ is integral over $\K[\sZ^{\pm1}]$, it follows by \cite[Sec. V.2.1, Corollary 4]{BBK}) that we can extend $\varphi$ to a morphism (that we will also
denote $\varphi$) from $\mathcal{A}=\K[\sZ^{\pm1}][A_\Ga]$ to $\R$
by setting $\varphi(A_\Ga)=\lambda= \sum_{i=1}^n \varphi(a_{i,1})\pi_{i}$. 
Therefore since $\mathcal{A}_{e}$ is a subalgebra of $\K[\sZ^{\pm1},q][A_\Ga]$, we get by restriction a morphism $\varphi:\mathcal{A}%
_{e}\rightarrow\mathbb{R}$. Applying $\varphi$ to the relation 
$$qA_\Ga\mb_{j}=q\sum_{i=1}^{n}a_{i,j}\mb_{i}$$
gives $\la \varphi(\mb_{j}) = \sum \varphi(a_{i,j})\varphi(\mb_i)$. It follows that $(\varphi
(\mb_{1}),\ldots,\varphi(\mb_{n}))$ is the left eigenvector of
$A_{\boldsymbol{t}}$ associated to~$\lambda$. Further $\mb_1=1$ so $\varphi
(\mb_{1})=1$. Since the normalised Perron-Frobenius vector is unique, we must have
$\varphi(\mb_{i})=\pi_{i}$ for all $i=1,\ldots,n$. This proves that we can
indeed define from any $\boldsymbol{t}=(t_{1},\ldots,t_{m})\in{\R_{+}^\ast}^{N}$ a morphism $\varphi:\mathcal{A}_{e}\rightarrow\mathbb{R}$ positive
on the basis $\mB_{\se}$ and such that $\varphi(\gamma)=1.$
\end{proof}
%

\begin{remark}
The arguments used in the proof of the previous theorem show that the
morphisms obtained (and thus also the extremal harmonic functions) are essentially
determined by their restrictions to $\K[\sZ^{\pm1}]$ thanks to the
Perron-Frobenius theorem. Nevertheless, in the previous construction, it can
happen that two elements in ${\R_{+}^\ast}^{N}$ give the same positive
harmonic function so that we do not get a complete parametrisation of the extremal harmonic functions.
\end{remark}

\section{The group of maximal indices}

\label{Sec_Group}We shall assume in this section that $\sZ=\emptyset$ and
$\K(\sZ^{\pm1})=\mathbb{C}$ (thus we only consider $\C$-algebras). All the graphs in this section belong to $\mathsf{Graph}_{n}(\R_+)$.
We have seen that in general, the property of a
finite graph to be positively multiplicative depends on a choice of the vertex considered as a root. Let $\Ga\in \sG{n}{\R_+}$ with vertices $\{v_1,\ldots,v_n\}$ be a positively multiplicative graph at $v_{1}$ with respect to the matrix realisation $(\cA,\mB)$ where $\mB=(\mb_1,\ldots,\mb_n)$. We further assume that $\cA$ is \emph{commutative}. Consider the cone
\[
\mathcal{C}(\mB)=\bigoplus_{j=1}^{n} \mathbb{R}_{+}\mb_{j}
\]
in the algebra $\mathcal{A}$. In this section, we will consider the set of
indices $i$ such that
\[
\mathcal{C}(\mB)(e_{i})=\mathbb{R}_{+}^n\qu{where $(e_1,\ldots,e_n)$ is the canonical basis of $\C^n$.} 
\]
We will prove in particular that it admits the structure of a commutative group. We
start with PM algebras coming from PM graphs of maximal dimension and next consider the general
situation of possibly infinite dimensional PM algebras.

\subsection{generalised permutations and PM graphs of maximal dimension}

Let $\Ga\in \sG{n}{\R_+}$ with vertices $\{v_1,\ldots,v_n\}$ be a PM graph at $v_1$ with respect ot the matrix realisation $(\cA,\mB)$ where $\mB=(\mb_1,\ldots,\mb_n)$. We will assume in this
section that $\Gamma$ is of maximal dimension so that
$\cA=\C[A_\Ga]$. Also, we can apply Theorem
\ref{ThGG2} to $\Gamma$ to see that the vector $e_{1}$ is cyclic and the map
$$
\begin{array}{cccccc}
\varphi&:&\C[A_\Ga]&\to&\oplus_{i=1}^{n}\mathbb{C}e_{i}\\
&&P(A_\Ga)&\mapsto&P(A_\Ga)\cdot e_{1}
\end{array}$$
 is an isomorphism of vector spaces. Moreover $\mb_{i}$ is the unique element in $\mathcal{A}$
such that $\mb_{i}(e_{1})=e_{i}$ for all $i$.

\begin{definition}
\label{def_maximal}
We say that a $i_{0}\in\{1,\ldots,n\}$ is \emph{maximal} 
if for
any $\mu=(\mu_1,\ldots,\mu_n)\in\mathbb{R}_{+}^{n}$ with $\mu_{1}+\cdots+\mu_{n}=1$ there
exists $\nu=(\nu_1,\ldots,\nu_n)\in\mathbb{R}_{+}^{n}$ such that%
\begin{equation*}
\sum_{i=1}^{n}\nu_{i}\mb_{i}(e_{i_{0}})=\mu\text{.} \label{Maximal}%
\end{equation*}
\end{definition}
Equivalently, $i_0$ is maximal if $\mathcal{C}(\mB)(e_{i_{0}})=\mathbb{R}_{+}^{n}$ but the definition above will be easier to extend to the case
of infinite dimensional algebras. It is easy to show that $1$ is maximal. Indeed, since $\mb_{i}(e_{1})=e_{i}$ for all $i=1,\ldots,n$, we have
\[
\sum_{i=1}^{n}\mu_{i}\mb_{i}(e_{1})=\mu\qu{ for all $\mu=(\mu_1,\ldots,\mu_n)\in \R^n$.}
\]

Recall from Section \ref{Subsec_Relabelling} that $\mathcal{G}_n$ is the group of generalised permutations of size $n$ and that we have  defined the subgroup
\[
G_{\Gamma}=\{U\in\mathcal{G}_{n}\mid UA_\Ga=A_\Ga U\}
\] 
of $\mathcal{G}_{n}$ which is the group of generalised automorphisms of the
graph $\Gamma$. 
The group $G_{\Gamma}$ is contained
in the centralizer of $A_\Ga$ which is equal to $\C[A_\Ga]$ (since $\Ga$ is of maximal dimension) and we have $G_{\Gamma}=\mathcal{G}_{n}\cap\mathcal{A}$. Consider $U$ in $G_{\Gamma}$ and set $U(e_{i})=\lambda_{i}e_{\sigma(i)}$ with $\lambda_{i}\in\R_{+}^\ast$  for $i=1,\ldots,n$ and $\sigma\in\mathfrak{S}_{n}$. We get
\[
\mb_{i}U(e_{1})=U\mb_{i}(e_{1})=U(e_{i})=\lambda_{i}e_{\sigma(i)}.
\]
Then
\[
\sum_{i=1}^{n}\frac{\mu_{\sigma(i)}}{\lambda_{i}}\lambda_{1}\mb_{i}%
(e_{\sigma(1)})=\sum_{i=1}^{n}\frac{\mu_{\sigma(i)}}{\lambda_{i}}%
\mb_{i}U(e_{1})=\mu\text{.}%
\]
This shows that $\sigma(1)$ is maximal and $\mb_{\sigma(1)}=\frac{1}%
{\lambda_{1}}U\in G_{\Gamma}$ because $U$ belongs to $\mathcal{A}$ and
$\frac{1}{\lambda_{1}}U(e_{1})=e_{\sigma(1)}$. Let us denote by $I_{m}\subset \{1,\ldots,n\}$ the
set of maximal indices for $\Gamma$. The results that we obtain in this section (working with countable infinite dimensional PM algebras) will imply the following statements:
\begin{itemize}
\item[{\tiny $\bullet$}] for all $i_0\in I_m$, $e_{i_0}$ is a cyclic vector for $A_\Ga$. Thus there exists a matrix realisation $(\cA_\Ga,\mB^{(i_0)})$ of $\Ga$  where $\mB^{(i_{0})}=\{\mb_{1}^{(i_{0})},\ldots,\mb_{n}^{(i_{0})}\}$ and $\mb^{(i_0)}_{i_0}=1$,
\item[{\tiny $\bullet$}]  for all $i_0\in I_m$, the basis $\mB^{(i_{0})}$ is
contained in the group $G_{\Gamma}$,

\item[{\tiny $\bullet$}]  the group $G_{\Gamma}$ acts transitively on $I_{m}$,

\item[{\tiny $\bullet$}]  the set $I_{m}$ itself has the structure of an abelian group.
\end{itemize}
%

\subsection{generalised adjacency algebra and positivity}

Consider a positively multiplicative commutative algebra $\mathcal{A}$ with respect to the countable basis $\mB=\{\mb_{i}\mid i\in
I\}$. Let $A_{i} = (c_{i,j}^k)_{j,k\in I}$ be the (possibly infinite) matrix of the multiplication $\mb_{i}$ expressed in the basis $\mB$ (that is $\mb_i\mb_j=\sum_k c_{i,j}^k\mb_k$). Then $A_i$ can be regarded as the
adjacency matrix of an oriented weighted graph with set of vertices $\{x_i\mid i\in I\}$ and edge weight function $\om$ defined by $\om(x_j,x_k)=c_{i,j}^k$ for all $j,k$ in $I$ (recall our convention that $\om(x,x')=0$ when there is no edge from $x$ to $x'$). When $n=\mathrm{card}(I)$ is
finite, the family of adjacency matrices $\{A_{i},i\in I\}$ generates a subalgebra of
$M_{n}(\mathbb{C})$ isomorphic to $\mathcal{A}$.

We can formalise this phenomenon by the notion of generalised adjacency
algebra. Throughout this section, $I$ is a countable set and $V$ will denote
the normed vector space $\ell^{1}(I)$ of complex sequences $(v_{i})$ indexed
by $I$ with the $\ell^{1}$-norm $\Vert v\Vert_{1}=\sum_{i\in I}|v_{i}|$. The
vector space $V$ is complete, and we identify $I$ with the Schauder
\footnote{A Schauder basis $S$ of a Banach space $V$ is a subset of $V$ such
that any element $v\in V$ as a unique writing as a convergent sum $\sum_{s\in
S}a_{s}s$ with $a_{s}\in\mathbb{C}$, $s\in S$. This is the natural extension
of the definition of the basis of a vector space in the Banach space setting.}
basis $(e_{i})_{i\in I}$ of $V$ given by $e_{i}(j)=\delta_{ij}$ for $i,j\in
I$. We write $V_{+}$ for the cone of $V$ consisting of complex sequences
taking non-negative real values, and $V_{+,1}$ the subset of $V_{+}$ having
$\ell^{1}$-norm equal to one. From a measure theory point of view, $V$ denotes
the set of finite complex measures on the set $I$, while $V_{+}$ (resp.
$V_{+,1}$) denotes the set of positive (resp. probability) measures on $V$.

We denote by $B(V)$ the set of endomorphisms of $V$ which are bounded with
respect to the $\ell^{1}$-norm, namely the set of linear maps $T:V\rightarrow
V$ such that
\[
\Vert T\Vert_{1,1}=\sup_{v\in V,\Vert v\Vert_{1}=1}\Vert Tv\Vert_{1}<+\infty.
\]
The map $T\mapsto\Vert T\Vert_{1,1}$ defines a norm on $B(V)$. If $T\in B(V)$,
we write $T_{ij}$ for the coefficient of $Te_{j}$ along $e_{i}$ for $i,j\in
I$. Given $T_{1}$ and $T_{2}$ in $B(V)$, we have%
\[
\left\Vert T_{1}T_{2}\right\Vert _{1,1}\leq\left\Vert T_{1}\right\Vert
_{1,1}\left\Vert T_{2}\right\Vert _{1,1}%
\]
so that $T_{1}T_{2}$ also belongs to $B(V)$. This shows that $B(V)$ is stable
by composition of operators.

\begin{notation}
Given any normed space $E$ and any subset $S\subset E$, write%
\[
\mathcal{C}(S)=\overline{\{\sum_{s\in S}\lambda_{s}s|\lambda_{s}\geq0,\,s\in
S\}}\text{ and }\mathcal{C}_{1}(S)=C(S)\cap B(0,1)
\]
where the closure is taken with respect to the norm on $E$ and $B(0,1)$ is the
unit ball.
\end{notation}

\begin{definition}
\ 

\begin{itemize}
\item An \emph{adjacency operator} on $I$ is an operator $T\in B(V)$ such
that
\[
T_{ij}\geq0
\]
for all $i,j\in I$.

\item An \emph{adjacency algebra} on $I$ is a closed commutative
unital subalgebra of $B(V)$ generated by adjacency operators on $I$.
\end{itemize}
\end{definition}

For an adjacency algebra $\mathcal{A}$ (regarded as a normed space as
previously), define the adjacency cone of $\mathcal{A}$ as the cone
\[
\mathcal{A}_+=\mathcal{C}(\{T\in\mathcal{A},T\text{ adjacency
operator on }I\}).
\]
By definition, $\mathcal{A}$ is then generated by $\mathcal{A}_+$.
Since $\mathcal{A}_+\cdot\mathcal{A}_+\subset\mathcal{A}_+$, this implies that
\begin{equation}
\mathcal{A}=\mathrm{span}(\mathcal{A}_+). \label{eq:span_A_C}%
\end{equation}

\begin{example}
Let $I=\{1,\ldots,n\}$ and let $V$ be $n$-dimensional $\C$-vector space with basis $\{e_{1},\ldots,e_{n}\}$. Let $A_\Ga$ is the adjacency matrix of an oriented graph $\Ga\in \sG{n}{\R_+}$. Then $\C[A_\Ga]$ is an adjacency algebra on~$I$. We always have
$\mathcal{C}(A^{k},k\geq0)\subset\C[A]_+$, but in general the inclusion
is strict. Consider for example
\[
A=%
\begin{pmatrix}
1 & 0\\
0 & 2
\end{pmatrix}
\text{ where }A-A^{2}/2=%
\begin{pmatrix}
1/2 & 0\\
0 & 0
\end{pmatrix}
\in\C[A]_+\setminus \mathcal C(A^{k},k\geq0).
\]

\end{example}

\begin{definition}
Two adjacency algebras $\mathcal{A},\mathcal{A}'$ on $I$ are
isomorphic if there exists an invertible adjacency operator $U:V\rightarrow V$
such that $U^{-1}$ is again an adjacency operator and such that $\mathcal{A}%
'=\{UAU^{-1}\mid A\in\mathcal{A}\}$.
\end{definition}

It is then straightforward to prove that for two isomorphic adjacency algebras
$\mathcal{A}$ and $\mathcal{A}'$ we have $\mathcal{A}
'_+=U\mathcal{A}_+U^{-1}$. As the following lemma shows,
isomorphisms between adjacency algebras are infinite-dimensional analogues of
the generalised permutations of the set $I$ (i.e. permutations and diagonal
rescaling of the basis $(e_{i})_{i\in I}$).

\begin{lemma}
\label{positive_inverse_implies_permutation} If $T$ is an invertible adjacency
operator such that $T^{-1}$ is also an adjacency operator, then $T$ is a
permutation up to multiplication by a diagonal operator with positive coefficients, meaning that there exists a
permutation $\sigma$ of $I$ and a family $\{\lambda_{i}\}_{i\in I}$ of
positive reals such that
\[
Te_{i}=\lambda_{i}e_{\sigma(i)}.
\]

\end{lemma}

\begin{proof}
Suppose by contradiction the existence of $i_{0}\in I$ and $i_{1},i_{2}\in I$
with $i_{1}\not =i_{2}$ such that
\[
Te_{i_{0}}=\lambda e_{i_{1}}+\mu e_{i_{2}}+\sum_{i\in I\setminus\{i_{1}%
,i_{2}\}}l_{i}e_{i}%
\]
with $\lambda,\,\mu>0$ and $l_{i}\geq0$ for all $i\in I\setminus\{i_{1}%
,i_{2}\}$. Then, since $T^{-1}Te_{i_{0}}=e_{i_{0}}$ and $T^{-1}$ has
non-negative entries, the latter equality implies that both $T^{-1}e_{i_{1}}$
and $T^{-1}e_{i_{2}}$ belong to $\mathbb{R}_{\geq0}e_{i_{0}}$. But this
contradicts the fact that $T^{-1}$ is invertible.
\end{proof}

For $i\in I$ and an adjacency algebra $\mathcal{A}$, denote by $\mathrm{ev}%
_{i}:\mathcal{A}\rightarrow V$ the evaluation map $A\mapsto Ae_{i}$. In the
following definition, recall that $\mathcal{A}$ is a Banach space as a closed
subspace of $B(V)$ and that a linear map $f:E\rightarrow F$ between Banach
spaces is \emph{coercive} if $\inf_{x\in E,\Vert x\Vert_{E}=1}\Vert
f(x)\Vert_{F}>0$. Remark in particular that coercivity implies injectivity for
a linear map. Both notions coincide when $I$ is finite because $\mathcal{A}$ is
then finite-dimensional and its unit ball is compact.

\begin{definition}
Let $\mathcal{A}$ be an adjacency algebra $\mathcal{A}$ on $I$ and consider an
element $i\in I$.

\begin{itemize}
\item The element $i$ is \emph{nondegenerate} if $\mathrm{ev}_{i}$ is coercive
(and thus injective) as a linear map from $\mathcal{A}$ to $V$.

\item The element $i$ is \emph{maximal} if $\mathrm{ev}_{i}$ is surjective
from $\mathcal{A}_+$ to $V_{+}$.
\end{itemize}
\end{definition}

In the latter definition, coercivity can be replaced by injectivity when $I$
is finite. Observe also that when $i$ is maximal,  by
\eqref{eq:span_A_C} the vector $e_i$ is cyclic for the algebra $\mathcal{A}$, that is we have $\mathcal{A}e_i=V$. 
\\We first state a result proving that we mostly have to establish
nondegeneracy at one maximal element to get nondegeneracy at any maximal
element. 

\begin{lemma}
\label{Lemma_propagation}Suppose that $i_{0}\in I$ is nondegenerate and
maximal. Then, any $i\in I$ which is maximal is also nondegenerate.
If $I$ is finite and $i_{0}\in I$ is maximal for $\mathcal{A}$, then $i_{0}$
is also nondegenerate and $\dim\mathcal{A}=|I|$.
\end{lemma}

\begin{proof}
Suppose that $i_{0}\in I$ is maximal and nondegenerate for $\mathcal{A}$, and
set
\[
\alpha=\inf_{T\in\mathcal{A},\Vert T\Vert_{1,1}=1}\Vert\mathrm{ev}_{i_{0}%
}(T)\Vert_{1}.
\]
We have $\alpha>0$ by coercivity of $\mathrm{ev}_{i_{0}}$. Let $i\in I$ be
another element which is also maximal. Let $T\in\mathcal{A}_+$ be
such that $\Vert T\Vert_{1,1}=1$. Since $e_{i}$ and $e_{i_{0}}$ are maximal,
there exist $U\in\mathcal{A}_+$ such that $Ue_{i}=e_{i_{0}}$ and
$U'\in\mathcal{A}_+$ such that $U'e_{i_{0}}=e_{i}$.
Hence, $UU'e_{i_{0}}=e_{i_{0}}$, and by nondegeneracy of $e_{i_{0}}$
we have $UU'=Id=U'U$ (recall that $\mathcal{A}$ is
commutative). In particular, $\Vert TU'\Vert_{1,1}\geq\Vert
U\Vert_{1,1}^{-1}\Vert T\Vert_{1,1}$. Hence,
\[
\Vert\mathrm{ev}_{i}(T)\Vert_{1}=\Vert\mathrm{ev}_{i_{0}}(TU')
\Vert_{1}\geq\alpha\Vert TU'\Vert_{1,1}\geq\alpha\Vert U\Vert
_{1,1}^{-1}\Vert T\Vert_{1,1},
\]
so that $\inf_{T\in\mathcal{A},\Vert T\Vert_{1,1}=1}\Vert\mathrm{ev}%
_{i}(T)\Vert_{1}\geq\alpha\Vert U\Vert_{1,1}^{-1}>0$, and $\mathrm{ev}_{i}$
is also coercive. Hence, $i$ is also nondegenerate for $\mathcal{A}$.

Suppose that $I$ is finite of cardinal $n$, so that $V$ is finite dimensional.
Then, $\mathcal{A}$ is a subalgebra of $M_{n}(\mathbb{C})$. By maximality at
$i_{0}$, the map $\mathrm{ev}_{i_{0}}$ is surjective from $\mathcal{A}$ to
$V$, which yields $\dim(\mathcal{A})\geq\dim V=n$. Suppose that $T\in
\mathcal{A}$ is such that $Te_{i_{0}}=0$. By surjectivity of $\mathrm{ev}%
_{i_{0}}$, for all $v\in V$ there exists $T_{v}\in\mathcal{A}$ such that
$T_{v}e_{i_{0}}=v$. Then, by commutativity of $\mathcal{A}$,
\[
Tv=TT_{v}e_{i_{0}}=T_{v}Te_{i_{0}}=0.
\]
Hence, $Tv=0$ for all $v\in V$, and so $T=0$, which implies that
$\mathrm{ev}_{i_{0}}$ is injective and $\dim(\mathcal{A})\geq n$. Thus
$\dim\mathcal{A}=n$ and $\mathrm{ev}_{i_{0}}$ is a linear isomorphism, in
particular it is coercive.
\end{proof}

\begin{remark}
In general, a commutative subalgebra of $M_{n}(\mathbb{C})$ can
have a dimension much bigger than $n$: by a theorem of Schur, such an algebra
can have dimension at most $\lfloor n^{2}/4\rfloor+1$, the bound being sharp.
Hence, by the latter lemma, a finite-dimensional adjacency algebra for which
there exists a maximal element is much closer to the case of a commutative subalgebra of $M_{n}(\mathbb{C})$ stable by the adjoint involution (see (4) of Example
\ref{Ex_PMalgebra}). Beware however that an adjacency algebra is not
necessarily diagonalisable (i.e isomorphic to $\mathbb{C}^n$), even if there exists a maximal element. This is
for example the case for the algebra of upper triangular Toeplitz matrices%
\[
\mathcal{A}=\left\{
\begin{pmatrix}
\lambda_{1} & \lambda_{2} & \dots & \lambda_{n}\\
0 & \lambda_{1} & \ddots & \vdots\\
\vdots &  & \ddots & \lambda_{2}\\
&  & \dots & \lambda_{1}%
\end{pmatrix}
,\lambda_{1},\ldots,\lambda_{n}\in\mathbb{C}\right\}  ,
\]
for which $e_{n}$ is maximal.
\end{remark}

\begin{proposition}
\label{prop:maximal_equal_positive_cone} Suppose that $\mathcal{A}$ is an
adjacency algebra for which $i_{0}$ is nondegenerate and maximal. Then, there
exists a unique basis $\mB=\{\mb_{i}\mid i\in I\}\subset\mathcal{A}_+$
of $\mathcal{A}$ such that $\mathcal{A}_+=\cC(\mB)$ and $\mb_{i}e_{i_{0}}=e_{i}$ for all~$i\in I$. In particular,
we have $b_{i_0}=1$ and $\cA$ is positively multiplicative.
\end{proposition}

\begin{proof}
Suppose that $i_{0}$ is nondegenerate and maximal for $\mathcal{A}$, and
denote by $\mb_{i}$ the element of~$\mathcal{A}_+$ such that
$\mb_{i}e_{i_{0}}=e_{i}$. Remark that $\{\mb_{i}\mid i\in I\}$ is a free family,
since it acts freely on $e_{i_{0}}$ by nondegeneracy. Let $T\in
\mathcal{A}_+$. Then, $Te_{i_{0}}=\sum_{i\in I}\lambda_{i}e_{i}$,
with $\lambda_{i}\geq0$ because $T$ is an adjacency operator, and
\[
\sum_{i\in I}\lambda_{i}=\left\Vert Te_{i_{0}}\right\Vert _{1}\leq\Vert
T\Vert_{1,1}\left\Vert e_{i_{0}}\right\Vert _{1}=\Vert T\Vert_{1,1}<+\infty.
\]
By nondegeneracy of $i_{0}$, the map $\mathrm{ev}_{i_{0}}$ is coercive,
implying that $\inf_{T\in\mathcal{A},T\not =0}\frac{\Vert Te_{i_{0}}\Vert_{1}%
}{\Vert T\Vert_{1,1}}=\alpha>0$. Hence, for $i\in I$ we have $\Vert
\mb_{i}e_{i_{0}}\Vert_{1}\geq\alpha\Vert \mb_{i}\Vert_{1,1}$. On the other hand by
construction, $\Vert \mb_{i}e_{i_{0}}\Vert_{1}=\Vert e_{i}\Vert_1=1$ for $i\in I$.
Hence, for all $i\in I$ we have $\Vert \mb_{i}\Vert_{1,1}\leq\frac{1}{\alpha}$.

Hence, $T':=\sum_{i\in I}\lambda_{i}\mb_{i}$ satisfies $\Vert T^{\prime
}\Vert_{1,1}\leq\frac{1}{\alpha}\sum_{i\in I}\lambda_{i}<+\infty$ and thus is
a well-defined element of $\mathcal{A}_+$ which satisfies
$T'e_{i_{0}}=Te_{i_{0}}$. By nondegeneracy, $T'=T$, and thus
$T\in\mathcal{C}(\mb_{i},\,i\in I)$. Therefore, $\mathcal{A}_+=\mathcal{C}(\mb_{i},\,i\in I)$. For $i,i'\in I$, we have
$\mb_{i}\mb_{i'}\in\mathcal{A}_+$ and by the previous result
$\mb_{i}\mb_{i'}=\sum_{i^{\prime\prime}\in I}\lambda_{ii'%
}^{i^{\prime\prime}}\mb_{i^{\prime\prime}}$ for some nonnegative coefficients $\lambda
_{ii'}^{i^{\prime\prime}}$.\ Therefore the basis $\{\mb_{i}\}_{i\in I}$
is positively multiplicative. We deduce that $\mathcal{A}$ is positively multiplicative. Finally, remark that the basis $\{\mb_i\}_{i\in I}$ is uniquely defined due to the injectivity of $\mathrm{ev}_{i_0}$.
\end{proof}
The proposition above yields an infinite dimensional version of Corollary  \ref{unique_normalised_basis}. 

\begin{example}
\label{ex:group_algebra} Given a discrete commutative group $G$, define
$$\ell^{1}(G)=\{f:G\rightarrow\mathbb{C}\mid\sum_{g\in G}|f(g)|<+\infty\},$$ and
consider the group algebra $\mathbb{C}[G]=\bigoplus_{g\in G}\mathbb{C}g$ with
multiplication given by the group structure. Consider the left-regular
representation $\rho$ of $G$ on the basis $\{\delta_{g},g\in G\}$ of $\ell
^{1}(G)$ such that
\[
\rho(g)\delta_{g'}=\delta_{gg'}.
\]
Let $\mathcal{A}^{G}$ be the closure of $\rho(\mathbb{C}[G])$ in $B(\ell
^{1}(G))$ with respect to the $\Vert\cdot\Vert_{1,1}$ norm. The algebra
$\mathcal{A}^{G}$ is then an adjacency algebra $G$, and the adjacency cone is
 $\mathcal{A}^{G}_+=\bigoplus_{g\in G}\mathbb{R}_{>0}\rho(g)$.
Indeed, each operator $\rho(g)$ acts by permutation and thus is an adjacency
operator. Moreover, since~$\rho(g)\rho(g')=\rho(gg')$ for
$g,g'\in G$, we indeed have $\mathrm{span}(\mathcal{A}^{G}_+)=\mathcal{A}_{G}$. Remark that every element $\delta_{g}$ with $g\in G$
is maximal and nondegenerate for $\mathcal{A}_{G}$.
\end{example}

We will see below that the example above is actually the only example of
adjacency algebra on a countable set $I$ for which every element of $I$ is
nondegenerate and maximal, up to a normalization. We will first prove a more
general result in Theorem \ref{prop:maximal_nondege_group}. 
\\
Since $I$ is
countable, any permutation $\sigma$ of $I$ extend linearly to an operator of
$B(V)$ (also denoted by $\sigma$) with the formula
\[
\sigma(e_{i})=e_{\sigma(i)}%
\]
and we have $\Vert\sigma\Vert_{1,1}=1$.

So, let $\mathcal{A}$ be an adjacency algebra on $I$, and denote by
$I_{m}\subset I$ the set of elements which are \emph{maximal and
nondegenerate}. We will assume that $I_{m}$ is not empty (thus contains all the maximal elements by Lemma \ref{Lemma_propagation}). Also denote by $\mathcal{G}$ the group of
linear automorphisms $T:V\rightarrow V$ such that $T(e_{i})=\lambda
_{i}e_{\sigma(i)}$ with $\lambda_{i}\in\mathbb{R}_{>0}$ and $\sigma$ any
permutation of~$I$. Recall that $G_{\mathcal{A}}$ is the subgroup $\mathcal{G}\cap\mathcal{A}$.

\begin{theorem}
\label{prop:maximal_nondege_group}Under the previous hypotheses, the following assertions hold. 
\begin{enumerate}
\item Each $i_{0}\in I_{m}$ defines a multiplicative basis $\mB%
^{(i_{0})}=\{\mb_{i}^{(i_{0})}\mid i\in I\}$ for the algebra $\mathcal{A}$.

\item For any $i_{0}\in I_{m}$, the set $\mB^{(i_{0})}=\{\mb_{i}%
^{(i_{0})}\mid i\in I\}$ is contained in the group $G_{\mathcal{A}}$.

\item The group $G_{\mathcal{A}}$ acts on $I_{m}$ transitively.

\item The set $I_{m}$ has the structure of a commutative group.\ Moreover,
there exits a representation $\rho$ of $I_{m}$ on $V$ such that $\{\rho(i)\mid
i\in I_{m}\}\subset\mathcal{A}_+$ and $\rho(i)e_{j}\in
\mathbb{R}_{>0}e_{i\cdot j}$ for $i,j\in I_{m}$, where $(i,j)\mapsto i\cdot j$
denotes the product structure on $I_{m}$.
\end{enumerate}
\end{theorem}

\begin{example}
We refer to \S~\ref{SubsecKR_Row2} for the example of the fusion
rule for the affine group $\widehat{su}(2)$. The  set  $I_{m}$ and its
associated group structure can be  easily made explicit. Indeed, the fusion rules shows that
$\mb_{0}\mb_{j}=1$ and $\mb_{l}\mb_{j}=\mb_{l-j}$ for any $j=0,\ldots,l$.\ Also when
$i=1,\ldots,l-1$ the action of $\mb_{i}$ does not yield a generalised
permutation of the basis $\mB$. Thus $I_{m}=\{0,l\}$ and the group
structure is given by $$l\cdot l=0, l\cdot0=0\cdot l=l, 0\cdot0=0.$$
\end{example}

The theorem above says in particular that, up to a nondegeneracy condition,
the set of maximal elements $I_{m}$ has the structure of an abelian group and,
up to a scaling, the action of the group algebra $\mathbb{C}[I_{m}]$ on
$\ell^{1}(I_{m})$ (as introduced in Example \ref{ex:group_algebra}) extends to
a representation on all $\ell^{1}(I)$ whose image forms an adjacency
subalgebra of $\mathcal{A}$. Beware that the group structure on~$I_{m}$ is not
uniquely defined and depend on the choice of an element in $I_{m}$.


\begin{proof}
[Proof of Theorem \ref{prop:maximal_nondege_group}] As already observed, the existence of a maximal nondegenerate element and Lemma \ref{Lemma_propagation} imply that $I_{m}$ is the set of maximal elements for $\mathcal{A}$. For $i\in
I_{m},i'\in I$, let $\mb_{i,i'}\in\mathcal{A}_+$ be
such that~$\mb_{i,i'}e_{i}=e_{i'}$. Then, if $i'\in
I_{m}$, there exists $\mb_{i',i}\in\mathcal{A}_+$ such that
$\mb_{i',i}e_{i'}=e_{i}$. This implies that~$\mb_{i'%
,i}\mb_{i,i'}e_{i}=e_{i}$, and by nondegeneracy of $e_{i}$ and the fact
that $\text{Id}\in\mathcal{A}_+$, this in turn implies that $\mb_{i^{\prime
},i}\mb_{i,i'}=\mb_{i,i'}\mb_{i',i}=\text{Id}$. Since
$\mb_{i,i'}$ and $\mb_{i',i}$ are adjacency operators, Lemma
\ref{positive_inverse_implies_permutation} yields that $\mb_{i,i'}$ and
$\mb_{i',i}$ belong to $G_{\mathcal{A}}$. Hence, for each $i,i^{\prime
}\in I_{m}$, there exists a infinite permutation matrix $\Sigma_{i,i'}$ and an infinite
diagonal matrix $D_{i,i'}$ with positive diagonal entries such that
\begin{equation}
\mb_{i,i'}=\Sigma_{i,i'}D_{i,i'}.
\label{eq:expression_permutation_dilation}%
\end{equation}
Fix an element $i_{0}\in I_{m}$. We get the basis
$\mB^{(i_{0})}$ by Proposition \ref{prop:maximal_equal_positive_cone}
such that $\mb_{i}(e_{i_{0}})=e_{i}$ for any $i\in I$. When $i$ belongs to $I_m$, we have by the previous arguments $\mb_{i_0,i}=\Sigma_{i_0,i}D_{i_0,i}$ and we can set $\mb_{i}=\mb_{i_{0},i}$, $\Sigma
_{i}=\Sigma_{i_{0},i}$ and $D_{i}=D_{i_{0},i}$. We can write
\begin{equation}
\mb_{i}\mb_{i'}=\sum_{i^{\prime\prime}\in I}\lambda_{i^{\prime\prime}%
}\mb_{i^{\prime\prime}}, \label{eq:multiplication_Tv_Tv'}%
\end{equation}
for all $i,i'\in I_{m}$ with $\lambda_{i^{\prime\prime}}\in
\mathbb{R}_{\geq0}$ for any $i^{\prime\prime}\in I$. Since $\mb_{i}$ and
$\mb_{i'}$ belong to the group $G_{\mathcal{A}}$, we also have
$\mb_{i}\mb_{i'}$ in $G_{\mathcal{A}}$. Therefore, by looking to the
expression on the right hand side of the equality, there should exist $j\in I$
such that $\mb_{i}\mb_{i'}=\mu \mb_{j}$ for some $\mu>0$. Indeed $\mb_i\mb_{i'}$ must send the vector $e_{i_0}$ on a scalar multiple of a vector $e_j$ with $j\in I$ because $b_i$ and $\mb_{i'}$ belong to $G_{\mathcal{A}}$. This is only possible when all but one of the coefficients $\lambda_{i^{\prime\prime}}$ are equal to zero. Recall that
$\mb_{j}(e_{i_{0}})=e_{j}$ and~$i_{0}\in I_{m}$. It follows that we also have
$j\in I_{m}$ since $\mathcal{A}_+b_j \subset \mathcal{A}_+$
and we can set $i\cdot i':=j\in I_{m}$ and~$\lambda
_{i,i'}=\lambda_{j}>0$ to get
\[
\mb_{i}\mb_{i'}=\lambda_{i,i'}\mb_{i\cdot i'}.
\]
From the commutativity and the associativity of the product in $\mathcal{A}$,
we deduce that the product $(i,i')\mapsto i\cdot i'$ is
commutative and associative. There is also a neutral element given by $i_{0}$,
since~$\mb_{i_{0}}=\text{Id}$. Likewise, expending $\mb_{i,i_{0}}$ on the basis
$\{\mb_{i}\mid i\in I\}$ and using a similar argument as before yields an element
$i'\in I$ such that
\[
\mb_{i,i_{0}}=\mb_{i}^{-1}=\lambda \mb_{i'}.
\]
Hence, this implies
\[
\mb_{i_{0}}=\text{Id}=\lambda \mb_{i'}\mb_{i}=\lambda\lambda_{i'%
,i}\mb_{i'\cdot i},
\]
and $i'\cdot i=i_{0}$. Finally, $I_{m}$ has a commutative group
structure given by the product $(i,i')\mapsto i\cdot i'$. Let
us denote by $G_{I_{m}}$ this group. 

Set $H=\{\lambda \mb_i\mid  \lambda>0, i\in I\}\subset \mathcal{A}$. By the above reasoning, $H$ is a subgroup of the group of invertible elements of $\mathcal{A}$ with multiplication $(\lambda \mb_i)\cdot(\mu \mb_{i'})=(\lambda\mu\lambda_{i,i'})\mb_{i\cdot i'}$. Moreover, $
\mathbb{R}_{>0}\mb_{i_0}$ is a divisible subgroup of $H$, which means that for any $n\in\mathbb{N}$, the endomorphism $\lambda \mb_{i_0}\mapsto (\lambda \mb_{i_{0}})^n$ is surjective. Hence, by \cite[Thm. 21.2]{LFuc}, there exists a subgroup $F\subset H$ such that 
$$H=(\mathbb{R}_{>0}\mb_{i_0})\times F.$$
For $i\in I$, let $\bar{\mb}_i\in F,\mu_i\in \mathbb{R}_{>0}$ be such that $\mb_i=(\mu_i\mb_{i_0})\bar{\mb}_i=\mu_i\bar{\mb}_i$. For $i,i'\in I$, on the one hand, $\bar{\mb}_i\bar{\mb}_{i'}\in F$, and on the other hand 
$$\bar{\mb}_i\bar{\mb}_{i'}=\frac{1}{\mu_i\mu_{i'}}\mb_i\mb_{i'}=\frac{\lambda_{i,i'}}{\mu_i\mu_{i'}}\mb_{i\cdot i'}=\frac{\lambda_{i,i'}\mu_{i\cdot i'}}{\mu_i\mu_{i'}}\bar{\mb}_{i\cdot i'}=\left(\frac{\lambda_{i,i'}\mu_{i\cdot i'}}{\mu_i\mu_{i'}}\mb_{i_0}\right)\bar{\mb}_{i\cdot i'}.$$
Since $H=(\mathbb{R}_{>0}\mb_{i_0})\times F$ and $\bar{b}_i\bar{b}_{i'}\in F$ we must have $\frac{\lambda_{i,i'}\mu_{i\cdot i'}}{\mu_i\mu_{i'}}=1$ and thus 
$$\bar{\mb}_i\bar{\mb}_{i'}=\bar{\mb}_{i\cdot i'}.$$
Hence, the map $\rho:i\mapsto\bar{\mb}_{i}$ yields a representation of
$G_{I_{m}}$ on $V$ such that $\rho(i)\in G_{\mathcal{A}}$ for all~$i\in I_{m}%
$. Moreover, for $i,j\in I_{m}$,
\[
\bar{\mb}_{i}e_{j}=\bar{\mb}_{i}\mb_{j}e_{i_0}=\mu_j\bar{\mb}_{i}\bar{\mb}_{j}e_{i_{0}}=\frac{\mu_j}{\mu_{i\cdot j}}\mb_{i\cdot
j}e_{i_{0}}=\frac{\mu_j}{\mu_{i\cdot j}}e_{i\cdot j}\in\mathbb{R}_{>0}e_{i\cdot
j},
\]
which implies that $\rho$ yields a free and transitive permutation
representation on the set of half-lines $\{\mathbb{R}_{>0}e_{i},i\in I_{m}\}$.
\end{proof}

As a corollary of Theorem \ref{prop:maximal_nondege_group}, we can
consider the case where all elements of~$I$ are maximal for
$\mathcal{A}$.

\begin{corollary}
\label{cor:group_structure} Suppose that $\mathcal{A}$ is an adjacency algebra
for which each $i\in I$ is maximal and nondegenerate. Then, there exists a
commutative group structure on $I$, and a positive diagonal operator $\Lambda$
such that $\Lambda\mathcal{A}\Lambda^{-1}=\mathcal{A}_{I}$, where as in
Example \ref{ex:group_algebra} the algebra $\mathcal{A}_{I}$ is the
norm-closure of the left-regular representation of $\mathbb{C}[I]$ on
$\ell^{1}(I)$ in the Banach algebra $B(\ell^{1}(I)).$
\end{corollary}

\begin{proof}
By Proposition \ref{prop:maximal_nondege_group}, $I$ has a group structure and
there is an action $\rho:I\rightarrow B(V)$ such that $\rho(i)\in\mathcal{A}$
for all $i\in I$ and $\rho(i)e_{j}\in\mathbb{R}_{>0}e_{i\cdot j}$ for all
$i,j\in I$. Let $i_{0}$ be the neutral element of~$I$ and set $\mb_{i}=\rho(i)$.
By the arguments used in the previous proof, for any $i\in I$ there exists
$\lambda_{i}>0$ such that $\mb_{i}e_{i_{0}}=\lambda_{i}e_{i\cdot i_{0}}%
=\lambda_{i}e_{i}$, with the particular case $\lambda_{i_{0}}=1$. Hence,
defining the diagonal operator $\Lambda$ by $\Lambda e_{i}=\lambda_{i}%
^{-1}e_{i}$, we have $\Lambda \mb_{i}\Lambda^{-1}e_{i_{0}}=e_{i}$. Hence, for
$i,j\in I$,
\[
\Lambda \mb_{i}\Lambda^{-1}e_{j}=\Lambda\rho(i)\Lambda^{-1}\Lambda\rho
(j)\Lambda^{-1}e_{i_{0}}=\Lambda\rho(i)\rho(j)\Lambda^{-1}e_{i_{0}}%
=\Lambda\rho(i\cdot j)\Lambda^{-1}e_{i_{0}}=e_{i\cdot j}.
\]
Let $\mathcal{A}_{I}$ be the algebra generated by $\Lambda \mb_{i}\Lambda^{-1}$,
$i\in I$. By the latter results, we have $\mathcal{A}_{I}=\rho(\mathbb{C}%
[I])$, where $\rho$ is the left regular representation introduced in Example
\ref{ex:group_algebra}. By Proposition \ref{prop:maximal_equal_positive_cone},
$\mathcal{A}$ is the norm closure of $\sum_{i\in I}\mathbb{C}b_{i}$. Hence,
$\Lambda\mathcal{A}\Lambda^{-1}$ is the norm closure of $\mathcal{A}_{I}$, and
thus we have
\[
\Lambda\mathcal{A}\Lambda^{-1}=\mathcal{A}_{I}.
\]

\end{proof}

\begin{example}
One can use Corollary \ref{cor:group_structure} to construct finite PM-graphs
with $I=I_{m}$. Consider for example the group algebra%
\[
\mathcal{A=}\mathbb{C}e_{(0,0)}\oplus\mathbb{C}e_{(1,0)}\oplus\mathbb{C}%
e_{(0,1)}\oplus\mathbb{C}e_{(1,1)}
\]
of $\mathbb{Z}/2\mathbb{Z\times Z}/2\mathbb{Z}$. Let $T$ and $T'$ be
the matrices of the multiplication by $e_{(1,0)}$ and $e_{(0,1)}$ expressed in
the previous basis and $(a,b)\in\mathbb{R}_{>0}^{2}$.\ Then%
\[
U=aT+bT'=\left(
\begin{array}
[c]{cccc}%
0 & a & 0 & 0\\
a & 0 & 0 & 0\\
0 & 0 & 0 & a\\
0 & 0 & a & 0
\end{array}
\right)  +\left(
\begin{array}
[c]{cccc}%
0 & 0 & b & 0\\
0 & 0 & 0 & b\\
b & 0 & 0 & 0\\
0 & b & 0 & 0
\end{array}
\right)  =\left(
\begin{array}
[c]{cccc}%
0 & a & b & 0\\
a & 0 & 0 & b\\
b & 0 & 0 & a\\
0 & b & a & 0
\end{array}
\right)  .
\]
Now choose $\lambda=(\lambda_{1},\lambda_{2},\lambda_{3},\lambda_{4}%
)\in\mathbb{R}_{>0}^{4}$ and set $V=DUD^{-1}$ where $D$ is the diagonal matrix
defined by $\lambda$. The matrix%
\[
V=\left(
\begin{array}
[c]{cccc}%
0 & a\frac{\lambda_{1}}{\lambda_{2}} & b\frac{\lambda_{1}}{\lambda_{3}} & 0\\
\frac{a}{\lambda_{1}}\lambda_{2} & 0 & 0 & b\frac{\lambda_{2}}{\lambda_{4}}\\
\frac{b}{\lambda_{1}}\lambda_{3} & 0 & 0 & a\frac{\lambda_{3}}{\lambda_{4}}\\
0 & \frac{b}{\lambda_{2}}\lambda_{4} & \frac{a}{\lambda_{3}}\lambda_{4} & 0
\end{array}
\right)
\]
is the adjacency matrix of a PM-graph such that $I=I_{m}$.$\allowbreak$
\end{example}

We end this section by showing that the set of positive roots in finite PM-graphs (See Definition \ref{Def_PMGraphs}) coincides with the set $I_{m}$ of maximal
indices in its associated basis $\mB$.

\begin{proposition}
Assume $k=\mathbb{C}$ and let $\Gamma$ be a finite PM graph rooted at $v_{i_{0}}$ with basis $\mB_{i_{0}}$ whose adjacency matrix has maximal dimension.

\begin{enumerate}
\item The set of maximal indices labels the positive roots of $\Gamma$, that
is $\mathcal{R}_{\Gamma}^{+}=\{v_{i}\mid i\in I_{m}\}$.

\item Set $\mB_{m,i_{0}}=\{\mb_{i}\in\mB_{i_{0}}\mid i\in
I_{m}\}$. Then, we have $\mB_{m,i_{0}}=\mB_{i_{0}}\cap
G_{\Gamma}$: the maximal indices labels the generalised permutations appearing
in $\mB_{i_{0}}$.
\end{enumerate}
\end{proposition}

\begin{proof}
Let us prove Assertion 1. The graph $\Gamma$ is multiplicative at $v_i$ for the basis $\mb_{i}^{-1}\mB_{i_{0}}$. Moreover since $\Gamma$ has maximal dimension, this multiplicative basis is unique. Therefore $v_{i}$ belongs to $\mathcal{R}_{\Gamma
}^{+}$ if and only if $b_{i}^{-1}\mB_{i_{0}}$ has nonnegative
structure constants. This is equivalent to say that $\mb_{i}^{-1}\mB%
_{i_{0}}\times \mb_{i}^{-1}\mB_{i_{0}}\subset \mb_{i}^{-1}\mathcal{A}_+$ because~$\mathcal{A}_+$ coincides with
$\mathcal{C}(\mB_{i_{0}})$ by Proposition
\ref{prop:maximal_equal_positive_cone} because $i_0 \in I_m$.
\\ For any $i\in I_{m}$, it follows from
Theorem \ref{prop:maximal_nondege_group} that $\mb_{i}$ is a generalised
permutation. Thus $\mb_{i}^{-1}\in\mathcal{A}_+$ and also $\mb_{i}%
^{-1}\mB_{i_{0}}\subset\mathcal{A}_+$.\ Now we have%
\[
\mb_{i}^{-1}\mB_{i_{0}}\times \mb_{i}^{-1}\mB_{i_{0}}=\mb_{i}%
^{-1}(\mb_{i}^{-1}\mB_{i_{0}}\times\mB_{i_{0}})\subset
\mb_{i}^{-1}(\mathcal{A}_+\times\mathcal{A}_+)\subset
\mb_{i}^{-1}\mathcal{A}_+
\]
because $\mathcal{A}_+$ is stable by product. This shows that
$\{v_i\,\vert\, i\in I_{m}\}\subset\mathcal{R}_{\Gamma}^{+}$. 
\\Conversely, assume that $i\in I$ is such that $v_i\in
\mathcal{R}_{\Gamma}^{+}.$ Then $\mb_{i}$ is invertible and $\mb_{i}^{-1}\in
\mb_{i}^{-1}\mB_{i_{0}}$ since $1\in \mB_{i_{0}}$.
Thus, by the previous arguments, we get $\mb_{i}^{-1}\times \mb_{i}^{-1}=\mb_{i}^{-1}c$ with $c\in
\mathcal{A}_+$.\ This proves that $\mb_{i}^{-1}\in\mathcal{A}_+$. Since the cone $\mathcal{A}_+$ is stable by
multiplication, it follows that the map $c\rightarrow \mb_{i}c$ is a bijection from $\mathcal{A}_+$ on itself (with
inverse the multiplication by~$\mb_{i}^{-1}$). This implies that
\[
\mathcal{A}_+(e_{i})=\mathcal{A}_+\mb_{i}(e_{i_{0}}%
)=\mb_{i}\mathcal{A}_+(e_{i_{0}})=\mathcal{A}_+(e_{i_{0}%
})=V_{+}.
\]
We get that $i$ is maximal.\ But it is also nondegenerate by Lemma
\ref{Lemma_propagation} because $\Gamma$ is finite. Thus, we have
$\mathcal{R}_{\Gamma}^{+}=\{v_i\,\vert\, i\in I_{m}\}$ as desired.

To prove Assertion 2, observe first that we have $\mB_{m,i_{0}%
}\subset\mB_{i_{0}}\cap G_{\Gamma}$ by Assertion 2 of Theorem~\ref{prop:maximal_nondege_group}. Now, if $\mb_{i}$ belongs to $\mB%
_{i_{0}}\cap G_{\Gamma}$, its inverse is an adjacency matrix in $\mathcal{A}$
and we have $\mb_{i}^{-1}\in\mathcal{A}_+$.\ By using the
previous arguments, we get $\mb_{i}\mathcal{A}_+=\mathcal{A}_+$ and therefore $\mathcal{A}_+v_{i}=\mathcal{A}_+\mb_{i}v_{i_{0}}=\mathcal{A}_+v_{i_{0}}=V_{+}$ which
shows that $i\in I_{m}$.\ Thus, $\mB_{m,i_{0}}\supset\mB%
_{i_{0}}\cap G_{\Gamma}$ as desired.
\end{proof}

\begin{remark}
	It also follows from the preceding results of this section that (for $k=\mathbb{C}$) a graph $\Gamma$ with $n$-vertices will be positively multiplicative as soon as its adjacency algebra $\mathbb{C}[A]$ can be embedded in a $n$-dimensional $\mathbb{C}$-algebra $\mathcal{A}$ generated by matrices with nonnegative entries for which there exists an index $1\leq i \leq n$ with $\mathcal{A}_+e_i=\bigoplus_{1\leq j\leq n}\mathbb{R}_{>0}e_j$.
\end{remark}

\end{document}